\documentclass[10pt,leqno,twoside]{amsart}

\usepackage{enumerate}
\usepackage{tikz}
\usepackage{pgfplots}
\usepackage{caption}
\usepackage{subcaption}
\usepackage{amssymb}
\usepackage{placeins}
\usepackage[english]{babel}
\usepackage{amssymb,amsthm,amsmath,eucal,mathrsfs}
\usepackage{bm}

\usepackage{amsmath}
\usepackage{amsfonts}
\usepackage{float}
\usepackage{amsmath}
\usepackage{amssymb}
\usepackage{graphicx}
\usepackage{lscape}
\usepackage{amstext}
\usepackage{amsthm}
\usepackage{color}
\usepackage{float}
\usepackage{mathrsfs}
\usepackage{epsfig}
\usepackage{url}
\usepackage{fancyhdr}
\usepackage{pspicture}
\usepackage{graphicx}

\usepackage{cite}

\usepackage{amsmath,verbatim}

\usepackage{amsthm}

\usepackage{amssymb}

\usepackage{amsfonts}

\usepackage{dsfont}

\usepackage{hyperref}

\newtheorem{theorem}{Theorem}[section]

\newtheorem{lemma}[theorem]{Lemma}

\newtheorem{corollary}[theorem]{Corollary}

\newtheorem{Hypotheses}{Hypotheses}

\theoremstyle{definition}

\newtheorem{remark}[theorem]{Remark}

\numberwithin{equation}{section}

\renewcommand{\Im}{{\ensuremath{\mathrm{Im\,}}}} 

\renewcommand{\Re}{{\ensuremath{\mathrm{Re\,}}}} 

\renewcommand{\div}{\mathrm{div}\,}    

\newcommand\restr[2]{{
  \left.\kern-\nulldelimiterspace 
  #1 
  \vphantom{\big|} 
  \right|_{#2} 
  }}

\usepackage{eucal}

\date{\today}

\title[Photo-acoustic inversion using plasmonics]{SIMULTANEOUS RECONSTRUCTION OF OPTICAL AND ACOUSTICAL
PROPERTIES IN PHOTO-ACOUSTIC IMAGING USING PLASMONICS
}

\author[Ghandriche and Sini]{Ahcene Ghandriche  $^*$ and Mourad Sini$^{\ddag}$}
\thanks{$^*$ RICAM, Austrian Academy of Sciences, Altenbergerstrasse 69, A-4040, Linz, Austria. Email: ahcene.ghandriche@ricam.oeaw.ac.at. This author is supported by the Austrian Science Fund (FWF): P 30756-NBL}
\thanks{$^{\ddag}$ RICAM, Austrian Academy of Sciences, Altenbergerstrasse 69, A-4040, Linz, Austria. Email: mourad.sini@oeaw.ac.at. This author is partially supported by the Austrian Science Fund (FWF): P 30756-NBL}

\begin{document}

\maketitle

\begin{abstract}
We propose an approach for the simultaneous reconstruction of the electromagnetic and acoustic material parameters, in the given medium $\Omega$ where to image, using the photo-acoustic pressure, measured on a single  point of the boundary of $\Omega$, generated by plasmonic nanoparticles. We prove that the generated pressure, that we denote by $p^{\star}\left(x, s, \omega \right)$, depending on only one fixed point $x \in \partial \Omega$, the time variable $s$, in a large enough interval, and the incidence frequency $\omega$, in a large enough band, is enough to reconstruct both the sound speed, the mass density and the permittivity inside $\Omega$. Indeed, from the behavior of the measured pressure in terms of time, we can estimate the travel time of the pressure, for arriving points inside $\Omega$, then using the eikonal equation we reconstruct the acoustic speed of propagation. In addition, we reconstruct the internal values of the acoustic Green's function. From the singularity analysis of this Green's function, we extract the integrals along the geodesics, for internal arriving points, of the logarithmic-gradient of the mass density. Solving this (internal) integral geometric problem provides us with the values of the mass density function inside $\Omega$. Finally, from the behavior of $p^{\star}\left(x, s, \omega \right)$ with respect to the frequency $\omega$, we detect the generated plasmonic resonances from which we reconstruct the permittivity inside $\Omega$.
\end{abstract}

\bigskip

\section{Introduction and statement of the results}
\subsection{Introduction of the mathematical model}
We deal with the photo-acoustic imaging modality. In this technique, we use a laser to excite a medium to image. As a consequence, the medium is heated and produces an elastic expansion that in turn generates a sound waves. This sound wave can be measured on the boundary, or in the exterior, of the medium. The goal of this technique is to reconstruct optical, and eventually acoustical, properties of the medium. Recently there has been much interest in this imaging technique which finds several applications in medical imaging. Without being exhaustive, we refer the reader to the following literature \cite{Habib-book, A-B-G-J:2012, B-U:2010,B-E-K-S:2018,B-G-S:2016, C-A-B:2007, H-B-P-S, Kirsch-Scherzer,K-K:2010,KuchmentKunyansky,N-S:2014,S:2010,S-U:2009, S-Y-2017,B-B-M-T}. In these works, the heat is created by a laser probe due to the presence of a pronounce absorption in the tissue to image. 

In the current work, we inject nanoparticles into the tissue that can create the heat while excited by lasers. 
Therefore, this technique has the potential to be applied also for tissues having very less absorption as themselves they create the needed contrast of absorption, i.e. they play the role of contrast agents. These nanoparticles can be either plasmonics or dielectrics. Both of them have small absorption (as compared to their diffusion character), i.e. the ratio between the imaginary part and the real part of the permittivity is small. However,  if we excite them with incident frequencies close to certain critical frequencies, called plasmonic or dielectric resonances, then the amount of electric (or magnetic) field generated will be enhanced and this will compensate the weakness of the absorption. As the source of the heat is generated by the product of the (imaginary part) of the permittivity and the square of the modulus of the electric field, see below, then the mentioned compensation makes sense. This is, in short, the principle of the Photo-Acoustic effect using resonating nano-particles.

The mathematical formulation of this imaging technique is based on the following equations 
\begin{equation*}
\left \{
\begin{array}{llrr}

Curl \circ Curl\; E -\omega^2\; \varepsilon\; \mu\; E=0,~~~ E:=E^s+E^i, \mbox{ in } \mathbb{R}^{3},\\

\rho_0 c_p\dfrac{\partial \bm{T}}{\partial t}-\nabla \cdot \kappa \nabla \bm{T} =\omega\; \Im(\varepsilon)\;\vert E \vert^2\; \delta_{0}(t),\; \mbox{ in } \mathbb{R}^{3}\times \mathbb{R}_+,\\

\dfrac{1}{c^2}\dfrac{\partial^2 p}{\partial t^2}-\rho^{-1}\nabla \cdot (\frac{1}{\rho}\nabla) p= \rho_0\; \beta_0\; \dfrac{\partial^2 \bm{T}}{\partial t^2}, \mbox{ in } \mathbb{R}^{3}\times \mathbb{R}_+,
\end{array} \right.
\end{equation*}
where $\omega$ is an incident frequency, $\varepsilon$ the electric permittivity, $\mu$ the magnetic permeability that we assume everywhere constant in $\mathbb{R}^{3}$, $E$ is the electric field, $\bm{T}$ the heat temperature, $p$ the acoustic pressure, $\rho$ (resp. $\rho_0$) is the mass density of the medium (reps. the background), $c_p$ the heat capacity, $\kappa$ is the heat conductivity, $c$ is the wave speed and $\beta_0$ is the thermal expansion coefficient. 
To the last two equations, we supplement the homogeneous initial conditions $
 \bm{T} =p=\dfrac{\partial p}{\partial t}=0,$ at $t=0$ and the Silver-M\"{u}ller radiation condition to $E^s$, see $(\ref{eq:electromagnetic_scattering})$. Under the condition that the heat conductivity is relatively small, the model above reduces to the following one
\begin{equation}\label{pressurwaveequa}
\left\{
\begin{array}{rll}
\partial^{2}_{t} p(x,t) - c^{2}(x) \, \rho(x) \underset{x}{\nabla} \cdot \left( \frac{1}{\rho(x)} \underset{x}{\nabla} p(x,t) \right) &=& 0, \quad \text{in} \quad \mathbb{R}^{3} \times \mathbb{R}^{+},\\
    p(x,0) &=& \dfrac{\omega \, \beta_{0}}{c_{p}} \; \Im(\varepsilon)(x) \; \vert E \vert^{2}(x), \qquad in \quad \mathbb{R}^{3}, \\ 
    \partial_{t}p(x,0) &=& 0, \qquad in \quad \mathbb{R}^{3}. 
    \end{array}
\right.
\end{equation}
More details on the actual derivation of this model can be found in \cite{P-P-B:2015, Triki-Vauthrin:2017} and more references therein. 
The source $E$ is solution of the scattering problem
\begin{equation}\label{eq:electromagnetic_scattering}
\left\{
\begin{array}{rll}
Curl \circ Curl \; E - \omega^2\; \varepsilon\; \mu\; E=0,~~~ E:=E^s+E^i, \mbox{ in } \mathbb{R}^{3},\\
\\
E^{i}(x) := d \, e^{i \, \omega \, \sqrt{\mu \, \epsilon_{\infty}} \, x \cdot \theta} \, \left\vert d \right\vert = \left\vert \theta \right\vert = 1, \, d \cdot \theta = 0, \\
\\
\underset{\left\vert x \right\vert \rightarrow + \infty}{\lim} \left\vert x \right\vert \, \left( \nabla \times E^s(x) \times \frac{x}{\left\vert x \right\vert} - \omega \, \sqrt{\mu \, \epsilon_{\infty}} \, E^s(x) \right) = 0,
\end{array}
\right.
\end{equation}
where the permittivity function $\varepsilon(\cdot)$ is defined as:
\begin{equation}\label{permittivityfct}
\varepsilon(\cdot) = \begin{cases}
\epsilon_{\infty} & \text{in} \quad \mathbb{R}^{3}\setminus \Omega, \\
\epsilon_{0}(\cdot) & \text{in} \quad  \Omega \setminus D, \\
\epsilon_{p} & \text{in} \quad D,
\end{cases}
\end{equation} 
with $\Omega$ as a bounded domain that represents the region where to do the imaging. 
Here $D \subset \Omega$ is the injected plasmonic nano-particle with permittivity $\epsilon_{p}$, permeability $\mu$, location $z$ and radius $a$, $D:=z+a\; B$ where $B$ contains the origin and is of maximum radius $1$. The parameter $a$ is taken to be small as compared to the maximum radius of $B$. The permeability $\mu$ is the same as the one of the background $\mu_\infty$. The permittivity $\epsilon_p$ at the nanoparticle  is related to the Lorentz model 
\begin{equation}\label{plasmonic}
\epsilon_p(\omega) := \epsilon_{\infty}\left( 1+\frac{\omega^2_p}{\omega^2_0-\omega^2+i \omega \gamma_p} \right)
\end{equation}
with $\omega_p$ as the electric plasma frequency, $\omega_0$ as the undamped frequency and $\gamma_p$ as the electric damping frequency which is assumed small, i.e. $0\leq \gamma_p \ll 1$, as $a\ll 1$.

Moreover, we assume that $\Im\left( \epsilon_{0}(\cdot) \right)$ is small,
i.e. 
\begin{equation}\label{SELZ}
\bm{\gamma} := \left\Vert \Im\left( \epsilon_{0}(\cdot) \right) \right\Vert_{\mathbb{L}^{\infty}\left(\Omega \right)}={ \scriptstyle \mathcal{O}}(1), \mbox{ as } a\ll 1.
\end{equation}
The order of the smallness of $\bm{\gamma}$ will be discussed later. The permittivity $\epsilon_{0}(\cdot)$ is variable and it is supposed to be smooth inside $\Omega$.
\bigskip
\newline
With such conditions, we show that $E$, the solution of (\ref{eq:electromagnetic_scattering}), is in $L^4(\Omega)$, see Section \ref{33RegE}.
\bigskip

The well-posedness of the problem $(\ref{pressurwaveequa})$ is investigated in the references \cite{salo} and \cite{smith}, where they proved that if the mass density $\rho$ is uniformly constant, under the condition of $\mathbb{L}^{2}(\mathbb{R}^{3})-$integrability of the source term, i.e. in our setting
\begin{equation*}
\frac{\omega \, \beta_{0}}{c_{p}} \, \Im\left( \varepsilon \right) \, \left\vert E \right\vert^{2} \in \mathbb{L}^{2}\left( \mathbb{R}^{3} \right), 
\end{equation*}
we have existence and uniqueness of a weak solution\footnote{For $\bm{X}$ a Banach space and $\bm{I} \subset \mathbb{R}$, the space $\mathcal{C}^{k}\left(\bm{I};\bm{X} \right)$ comprises all function $f : \bm{I} \rightarrow \bm{X}$ with $k$ continuous derivatives.} 
\begin{equation}\label{regP}
p(\cdot,\cdot) \in \mathcal{C}\left([-M;M]; \mathbb{L}^{2}\left( \mathbb{R}^{3} \right) \right) \cap \mathcal{C}^{1}\left([-M;M]; \mathbb{H}^{-1}\left( \mathbb{R}^{3} \right) \right),
\end{equation}
 solution of $(\ref{pressurwaveequa})$ in distributional sense, where the constant $M$ is chosen large such that:\footnote{We recall that: 
\begin{equation*}
\left\Vert f \right\Vert_{\mathcal{C}^{1,1}\left(\mathbb{R}^{3}\right)}:= \left\Vert f \right\Vert_{\mathcal{C}^{1}\left(\mathbb{R}^{3}\right)} + \underset{\left\vert \beta \right\vert = 1}{\max} \;\; \underset{x,y \in \mathbb{R}^{3} \atop x \neq y}{Sup} \;\; \frac{\left\vert \partial^{\beta} f(x) - \partial^{\beta} f(y) \right\vert}{\left\vert x-y \right\vert}. 
\end{equation*}
}
\begin{equation}\label{low-upper-condition-c}
\left\Vert c^{2}(\cdot) \right\Vert_{\mathcal{C}^{1,1}\left( \mathbb{R}^{3} \right)} \leq M \quad \text{and} \quad c(\cdot) \geq M^{-1}.
\end{equation}

\medskip

For a variable mass density $\rho$, the existence and uniqueness of the solution of the problem $(\ref{pressurwaveequa})$ are investigated in Subsection \ref{Eu}.



\subsection{Integral representation of the solution }\label{RG}

We associate to the problem $(\ref{pressurwaveequa})$, the Green's kernel solution, in the distributional sense, of  
\begin{equation}\label{pa2ndCGreen}
\begin{cases} 
\partial_{t}^{\prime\prime} G(x,t,y,s) - c^{2}(x) \,  \underset{x}{\Delta} \, G(x,t,y,s) \, + c^{2}(x) \, \nabla \log(\rho(x)) \cdot \underset{x}{\nabla} \, G(x,t,y,s) = \underset{y}{\delta}(x) \, \underset{s}{\delta}(t) , \quad \text{in} \quad \mathbb{R}^{3} \times \mathbb{R}^{+}, \\ 
\qquad \qquad \qquad \qquad G(x,0) = 0, \quad \text{in} \quad \mathbb{R}^{3}, \\
\qquad \qquad \qquad \quad \partial_{t} G(x,0) = 0, \quad \text{in} \quad \mathbb{R}^{3}.
\end{cases}
\end{equation}
Thanks to \cite{romanov2009smoothness}, formula (1.9) therein, we know that the solution to $(\ref{pa2ndCGreen})$ is represented as: 
\begin{equation}\label{Green'skernelvarspeed}
G(x,t,y) :=  \sum_{k = -1}^{+\infty} \alpha_{k}(x,y) \,\, \Theta_{k}\left( t^{2}-\tau^{2}(x,y)\right), \quad x \neq y, \quad t \geq 0,
\end{equation}
where the function $\tau(x,y)$ is the travelling time of a signal from a point $x$, in our case fixed on the boundary $\partial \Omega$, to a point $y \in \Omega$. Evidently, $\tau(x,y) = \tau(y,x)$. This function $\tau(\cdot,\cdot)$ defines a Riemannian metric with arc length $d\tau$ given by 
\begin{equation*}
d\tau = c^{-1}(y) \; \sqrt{\sum_{j=1}^{3} \, \left( d \, y_{j} \right)^{2}}.
\end{equation*}
  The sequence of functions $\Theta_{k}\left( \cdot \right), k \geq -1,$ is defined as follows: 
\begin{equation}\label{defThetak}
\Theta_{-1}(t) = \underset{0}{\delta}(t),\,\, \Theta_{0}(t) = \begin{cases}
      1, & \text{if $t \geq 0,$} \\
      0, & \text{if $ t < 0,$}
    \end{cases} \quad \text{and} \quad  \Theta_{k}(t) = \frac{t^{k}}{k!} \, \Theta_{0}(t), \,\, k \geq 1. 
\end{equation}
Moreover, the sequence of functions $\alpha_{k}(x,y), \, k \geq -1,$ are infinitely differentiable functions and are defined recurrently by
\begin{equation}\label{defalpha-1}
\alpha_{-1}(x,y) := \frac{1}{2 \, \pi} \left( \det \, \frac{\partial}{\partial \, x} \left( \frac{-1}{2} \left(\underset{y}{\nabla} \tau^{2}(x,y) \right)^{tr} \right) \right)^{\frac{1}{2}} \,\, \exp\left( \frac{1}{2} \, \int_{\Gamma(x,y)} \langle \underset{\xi}{\nabla} \log(\rho(\xi)) ; d\xi \rangle \right)
\end{equation}
and 
\begin{eqnarray}\label{defalphak}
\nonumber
\alpha_{k}(x,y) &:=& \frac{\alpha_{-1}(x,y)}{4 \, \, \tau(x,y)^{k+1}} \, \int_{\Gamma(x,y)}\,\frac{c^{2}(\xi) \, \underset{\xi}{\Delta} \,  \alpha_{k-1}(\xi,y)}{\alpha_{-1}(\xi,y)} \,\tau(\xi,y)^{k} \, d\tau(\xi,y) \\
&-& \frac{\alpha_{-1}(x,y)}{4 \, \, \tau(x,y)^{k+1}} \, \int_{\Gamma(x,y)}\,\frac{c^{2}(\xi) \, \nabla \log(\rho(\xi)) \cdot \underset{\xi}{\nabla} \,  \alpha_{k-1}(\xi,y)}{\alpha_{-1}(\xi,y)} \,\tau(\xi,y)^{k} \, d\tau(\xi,y), \,\, k \geq 0,
\end{eqnarray} 
where the function $\Gamma(x,y)$ represents the geodesic, in the metric $\tau(\cdot,\cdot)$, connecting point $y \in \Omega$ to point $x$ of the boundary $\partial \Omega$. The point $\xi = (\xi_{1},\xi_{2},\xi_{3})$, is on $\Gamma(x,y)$, represents the Riemannian coordinates and it is given by
$
\xi = - \frac{1}{2} \, c^{2}(y) \, \underset{y}{\nabla} \tau^{2}(x,y) \equiv h(x,y).$
\medskip
\newline
The following assumptions are needed to derive (\ref{Green'skernelvarspeed}), which will be used in the sequel, therefore, we state them as hypotheses.
\begin{Hypotheses}\label{Hyp}
The metric $\tau$, and also the families of geodesics $\Gamma(\cdot,\cdot)$, are taken satisfying the following properties:  
\begin{enumerate}
\item Any two points of the domain $\Omega$ are connected by a unique geodesic $\Gamma(\cdot,\cdot)$, of the metric $\tau$, contained in $\Omega$ and with ends points on the boundary $\partial \Omega$. 
\item[]
\item The boundary $\partial \Omega$ is convex relative to these geodesics.
\end{enumerate}
\end{Hypotheses}
\medskip


As shown in Section $\ref{AppendixOne}$ the solution of $(\ref{pressurwaveequa})$ can be expressed through the source term by the formula
\begin{equation}\label{SolutionPressure}
p(x,t) =  \frac{\omega \, \beta_{0}}{c_{p}} \,\, \partial_{t} \,\int_{\Omega} G(y,t,x) \,\, \Im\left( \varepsilon \right)(y) \,\, \left\vert E \right\vert^{2}(y)  \,\,  dy.
\end{equation}

To avoid making the text more cumbersome, we warn the reader that we omit to note the multiplicative constant  $\dfrac{\omega \, \beta_{0}}{c_{p}}$ in the integral expression of the solution given by $(\ref{SolutionPressure})$. In addition, to fix notations, in the presence of one particle, we set $E:=u_{1}$ and  we rewrite the representation  of the solution $p(\cdot,\cdot)$, as:
\begin{equation}\label{SolutionPressureI}
p(x,t) = \partial_{t} \,\int_{\Omega} G(x,t,y) \,  \, \Im\left( \varepsilon \right)(y) \, \left\vert u_{1} \right\vert^{2}(y)  \,  dy.
\end{equation}

\subsection{Statement of the results}
We need the following theorem to analyze the variation of the averaged  pressure after injecting single plasmonic nano-particles $D$ inside $\Omega$. 
\begin{theorem}\label{THMVarSpeedCvar}
Assume the permittivity of the medium $\epsilon_{0}$ to be of class $C^1(\Omega)$ and constant outside $\Omega$. In addition, we assume that the permittivity of the nano-particle to be given by $(\ref{plasmonic})$. 
Regarding the acoustic model, we assume
the sound speed $c$ and the mass density $\rho$ of class $C^\infty(\mathbb{R}^3)$ such that $(\ref{low-upper-condition-c})$ and Hypotheses \ref{Hyp} are satisfied. \footnote{The $C^{\infty}$ regularity of $c$ and $\rho$ is used to derive the singularity analysis (\ref{Green'skernelvarspeed}) of the Green's function of the wave equation in \cite{romanov2013integral}. It can be reduced considerably, see \cite{romanov2009smoothness}. In addition, to reconstruct the permittivity function $\epsilon_0$ and the speed of sound $c$, this expansion is not needed. It is only needed to reconstruct the mass density $\rho$.}
Under these conditions, we have the following approximations of the average pressure:
\begin{enumerate}
\item Before the entrance time, i.e.  $s < \tau_{1}(x, z):= \left(\underset{y \in D}{Inf} \,\, \tau(x,y) \right) - a$,
\begin{equation}\label{I1!}
p^{\star}(x,s) := \int_{0}^{s} 2 r \int_{0}^{r} p(x,t)  dt \, dr = \mathcal{O}\left( \bm{\gamma} \right). 
\end{equation}
\item After the exit time, i.e. $ s > \tau_{2}(x, z):=\left(\underset{y \in D}{Sup} \,\, \tau(x,y) \right) + a $,
\begin{equation}\label{I2!}
p^{\star}(x,s) = \Psi_{2}(x,z,s) \, \int_{D}   \left\vert u_{1} \right\vert^{2}(y) dy  + \mathcal{O}\left( a^{4-2h} \right) + \mathcal{O}\left( \bm{\gamma} \right), 
\end{equation}
where $z \in D$ and
\begin{equation*}
\Psi_{2}(x,z,s) :=  \Im\left( \epsilon_{p} \right) \,   \,  \alpha_{-1}(z,x) \, \left\vert \underset{y}{\nabla} \tau(z,x) \right\vert +  \Im\left( \epsilon_{p} \right) \, \int_{\tau_{2}(x, z)}^{s} 2 \, r \, \sum_{k=0}^{+\infty} \alpha_{k}(z,x)  \,  \frac{\left(r^{2} - \tau^{2}(z,x) \right)^{k}}{k !} \,\, dr.
\end{equation*}
\item Between the entrance and the exit time, i.e. $\tau_{1}(x, z) < s < \tau_{2}(x, z)$, 
\begin{equation*}
p^{\star}(x,s) = \Psi_{1}(x,z^{\star}) \, \int_{D \cap RV_{\tau}(x,s)} \left\vert u_{1} \right\vert^{2}(y) \, dy + \mathcal{O}\left( a^{4-2h} \right)  + \mathcal{O}\left( \bm{\gamma} \right),
\end{equation*}
where $z^{\star} \in D \cap RV_{\tau}(x,s)$ and 
\begin{equation*}
\Psi_{1}(x,z^{\star}) :=  \Im\left( \epsilon_{p} \right) \,   \,  \alpha_{-1}(z^{\star},x) \, \left\vert \underset{y}{\nabla} \tau(z^{\star},x) \right\vert.
\end{equation*}
Moreover, for $\varrho$ positive reel number, the subset $RV_{\tau}(x,\varrho)$ that we call Riemannian volume of center $x \in \mathbb{R}^{3}$ and radius $\varrho$, is defined by:
\begin{equation*}
RV_{\tau}(x,\varrho) := \{y \in \mathbb{R}^{3}, \quad \text{such that} \quad \tau(x,y) \leq \varrho  \}.
\end{equation*}  
\end{enumerate}
\end{theorem}
\bigskip
The next theorem provides a useful a-priori estimate and a precise dominating term of the electric field solution of the electromagnetic model in terms of the parameter $a$. To state these properties, we assume also here that $\epsilon_0(\cdot)$ to be of class $\mathcal{C}^1$. Let $z\in \Omega$ and define
\begin{equation*}
f_n(\omega, z):=\epsilon_0(z) - (\epsilon_0(z)-\epsilon_p(\omega)) \, \lambda_n
\end{equation*}
where $(\lambda_n)_{n\in \mathbb{N}}$ is the sequence of the eigenvalues of the Magnetization operator $\nabla M(\cdot)$ restricted to $\nabla \mathcal{H}armonic(D):=\{u=\nabla \phi,~~ \Delta \phi =0 \mbox{ in } D\}$. We recall that $\nabla M(\cdot)$ is  defined, from $\mathbb{L}^{p}\left( D \right)$ to $\mathbb{L}^{p}\left( D \right)$, where $1 \leq p \leq \infty$, by: 
\begin{equation*}
\nabla M(V)(x) := \underset{x}{\nabla} \int_{D} \underset{y}{\nabla} \Phi_{0}(x,y) \cdot V(y) \,\, dy := \underset{x}{\nabla} \int_{D} \underset{y}{\nabla} \left( \frac{1}{\left\vert x - y \right\vert} \right) \cdot V(y) \,\, dy, \,\, x \in D. 
\end{equation*}
We show that the dispersion equation $f_n(\omega, z)=0$ has one and only one solution in the complex plan, with a dominant part in the interval $\left( \omega_{0};~~\sqrt{ \omega^{2}_{p} + \omega^{2}_{0}} =: \omega_{max} \right) $. For any $n_0$ fixed, we set $\omega_{n_0}$ to be the corresponding solution for $n=n_0$. 
\begin{theorem}\label{Lemma-E} We assume that $\Omega$ and $D$ to be of class $C^2$ and $\epsilon_0$ of class $C^1$. 
Let the used incident frequency $\omega$ be such that 
\begin{equation*}
\omega^2-\omega^2_{n_{0}} \sim a^h, ~~ h\in [0, 1). 
\end{equation*}
The electric field $u_{1}(\cdot)$ satisfies the following a priori estimate 
\begin{equation}\label{aprioriestimate}
\left\Vert u_{1} \right\Vert_{\mathbb{L}^{2}(D)} \leq a^{-h} \,\, \left\Vert u_{0} \right\Vert_{\mathbb{L}^{2}(D)},
\end{equation} 
and the following approximation
\begin{equation}\label{NearRes}
\int_{D} \left\vert u_{1} \right\vert^{2}(x) \,\, dx = \frac{a^{3} \, \left\vert \epsilon_{0}(z) \right\vert^{2} \, \left\vert \langle u_{0}(z) ; \int_{B} e_{n_{0}}(x) \, dx \rangle  \right\vert^{2}}{\left\vert \epsilon_{0}(z) - (\epsilon_{0}(z) - \epsilon_{p}(\omega)) \,\, \lambda_{n_0} \right\vert^{2}} + \mathcal{O}\left(a^{\min(3;4-3h)} \right).
\end{equation}
\end{theorem}
\begin{proof}
For the proof, we refer the readers to Theorem 1.1 and Proposition 2.2 in \cite{AhceneMouradMaxwell}.
\end{proof}

Theorem \ref{THMVarSpeedCvar}, coupled with Theorem \ref{Lemma-E}, suggests the following corollary.  
\begin{corollary}\label{coro}
Let the conditions in Theorem \ref{THMVarSpeedCvar}  and Theorem \ref{Lemma-E} be satisfied. After the exit time, i.e. $s > \tau_{2}(x, z)$, and if
\begin{equation*}
\omega^2-\omega^2_{n_{0}} \sim a^h, ~~ h\in [0, 1) 
\end{equation*}
we have
\begin{equation}\label{1Coro}
p^{\star}(x,s,\omega) =  \Psi_{2}(x,z,s) \, \frac{a^{3} \, \left\vert \epsilon_{0}(z) \right\vert^{2} \, \left\vert \langle u_{0}(z) ; \int_{B} e_{n_{0}}(x) \, dx \rangle \right\vert^{2}}{\left\vert \epsilon_{0}(z) - \left(\epsilon_{0}(z) - \epsilon_{p}(\omega) \right) \,\lambda_{n_0} \right\vert^{2}} + \mathcal{O}\left( a^{\min(3,4-3h)} \right) + \mathcal{O}\left( \bm{\gamma} \right).
\end{equation}
with
\begin{equation*}
\langle u_{0}(z) ; \int_{B} e_{n_{0}}(x) \, dx \rangle = \sum_{m} \langle u_{0}(z) ; \int_{B} e_{n_{0},m}(x) \, dx \rangle,
\end{equation*}
where $e_{n_{0},m}(\cdot)$ are such that $\nabla M \left(e_{n_{0},m} \right) = \lambda_{n_0} \,\, e_{n_{0},m}$.

\end{corollary}
\begin{proof}
The previous corollary is a straightforward consequence of $(\ref{I2!})$ and the approximation of the electric field given by $(\ref{NearRes})$. \end{proof}

Observe that when $h=0$, i.e. the incident frequency $\omega$ is away from the resonance $\omega_{n_0}$, then 
\begin{equation*}
p^{\star}(x,s,\omega) = \mathcal{O}\left( a^{3} \right) + \mathcal{O}\left( \bm{\gamma} \right).
\end{equation*}

\subsection{Simultaneous reconstruction of $\rho$, $c$ and $\epsilon_{0}$}\label{Section-Reconstruction}
Based on Theorem \ref{THMVarSpeedCvar} and Corollary \ref{coro}, we propose a scheme to reconstruct $\rho$, $c$ and $\epsilon_{0}$ from the measurements of $p(x, s, \omega)$ for
\begin{enumerate}
\item a fixed point $x$ on the boundary of $\Omega$,
\item [] 
\item a large enough interval of time $s$, i.e. $s \in \left( 0,\;~ M^{-1}\; Diam(\Omega) \right)$, recalling that $M^{-1}< \underset{x \in \Omega}{\inf} c(x)$, see (\ref{low-upper-condition-c}), and $Diam(\Omega)$ stands for the diameter of $\Omega$, 
\item[]
\item and a large enough band of incident frequencies $\omega$, i.e. $\omega \in \left( \omega_{min}:=\omega_0,\;~ \omega_{max} :=\sqrt{ \omega^{2}_{p} + \omega^{2}_{0}} \right)$.
\end{enumerate}
\bigskip


We proceed as follows.
\bigskip

\begin{enumerate}
\item We look at the variation of the collected pressure $p^{\star}(\cdot, \cdot)$, as a function only with respect to time variable, i.e. regardless on the used frequencies, we see that $p^{\star}(\cdot, \cdot)$ before the entrance time \,\, $\tau_{1}(x, z) := \left(\underset{y \in D}{Inf} \,\, \tau(x,y) \right) - a$, \,\, to the particle $D$, is small of order $\mathcal{O}\left( \bm{\gamma} \right)$. Furthermore, after the exit time \,\, $\tau_{2}(x, z) := \left(\underset{y \in D}{Sup} \,\, \tau(x,y) \right) + a$, from the particle $D$, we observe two scenarios, see for instance Corollary \ref{coro} and the schematic representation for the average pressure given by Figure \ref{redgreen}. We observe that the average pressure $p^{\star}(\cdot, \cdot)$ is large compared to $\bm{\gamma}$. The observed jump, of course in continuous manner\footnote{The continuity of $p^{\star}(\cdot, \cdot)$, with respect to time variable, is a consequence of $(\ref{regP})$ and $(\ref{I1!})$.}, in the interval $]\tau_{1}(x, z);\tau_{2}(x, z)[$ allows us to reconstruct the travel time $\tau_{1}(x, z)$ and/or\footnote{As the size of the particle $D$ is of order \, $a$ \, and the function $\tau(x,\cdot)$ is smooth one, we cannot distinguish between $\tau_{1}(x, z)$ and $\tau_{2}(x, z)$, i.e. $\left\vert \tau_{1}(x, z) - \tau_{2}(x, z) \right\vert = \mathcal{O}\left( a \right)$.} $\tau_{2}(x, z)$. By moving the nano-particle $D$ inside $\Omega$, we reconstruct for fixed $x \in \partial \Omega$ the travel time function $\tau(x,\cdot)$. After the reconstruction of the travel time function, we use the eikonal equation,
see \cite[Chapter 3]{romanovBook}, given by:
\begin{equation*}
\left\vert \underset{z}{\nabla} \tau(
x,z) \right\vert^{-1} \, = \, c(z), \quad z \in D,
\end{equation*}
to reconstruct the speed of propagation, namely the function $c(\cdot)$ on the location point $z$ of the nanoparticle $D$. By moving $D$ inside omega we reconstruct the function $c(\cdot)$ inside the domain $\Omega$.  
\item[]
\item We have seen that from the behaviour of the map $s\rightarrow p^{\star}(x,s,\omega)$, for $x \in \partial \Omega$ and $\omega$ fixed, we can reconstruct the internal values of the travel time, i.e. $\tau(x, z)$ for $z\in \Omega$, and hence the acoustic speed $c(z)$ for $ z\in\Omega$. Now, we use the map $\omega \rightarrow p^{\star}(x,s,\omega)$, for $x \in \partial \Omega$ and $s> \tau_2(x, z)$ fixed, to reconstruct the permittivity function $\epsilon_0$. Indeed, based on the expansion $(\ref{1Coro})$, we see that this map will reach its maximum at the zeros of $\vert \epsilon_{0}(z) - \left(\epsilon_{0}(z) - \epsilon_{p}(\omega) \right) \,\lambda_{n_0}\vert$. We recall that for every $z\in \Omega$ and for every $n \in \mathbb{N}$, the dispersion equation 
\begin{equation*}
f_n(\omega, z):=\epsilon_0(z) - (\epsilon_0(z)-\epsilon_{p}(\omega))\lambda_n=0,
\end{equation*}
has one and only one solution in the complex plan, that we denote by $\omega_{n, \mathbb{C}}$, given by:
\begin{equation*}
\omega_{n, \mathbb{C}} = \frac{i \, \gamma_{p} \mp \sqrt{- \gamma_{p}^{2} + 4 \, \left( \omega^{2}_{0} + \dfrac{\epsilon_{\infty} \, \lambda_{n} \, \omega_{p}^{2}}{\epsilon_{\infty} \, \lambda_{n} + (1 - \lambda_{n}) \, \epsilon_{0}(z)} \right)}}{2},
\end{equation*}
with it's dominant part contained in the interval $\left( \omega_{0};~~\sqrt{ \omega^{2}_{p} + \omega^{2}_{0}} \right)$. In addition, we show that with
\begin{equation*}
\bm{\omega_{n_{0}}} :=  \left( \omega_{0}^{2}  + \dfrac{\omega_{p}^{2} \, \lambda_{n_{0}} \, \epsilon_{\infty}}{\lambda_{n_{0}} \, \epsilon_{\infty} + (1 - \lambda_{n_{0}}) \, \Re(\epsilon_{0}(z)) } \right)^{\frac{1}{2}},
\end{equation*}
we have, see Section \ref{Justification-omega-n-0},
\begin{eqnarray*}
 f_{n_{0}}\left(\bm{\omega_{n_{0}}}, z \right) = \mathcal{O}(\bm{\gamma})  + \mathcal{O}(\gamma_{p}).
\end{eqnarray*}
 
Consequently, from  $(\ref{1Coro})$, we deduce that $p^{\star}(x,s,\cdot)$ admits a peaks near $\bm{\omega_{n_{0}}}$.
Therefore by plotting the curve $\omega \rightarrow p^{\star}(x,s,\omega)$, for $x \in \partial \Omega$ and $s> \tau_2(x, z)$ fixed, in the interval $\left( \omega_{0};~~\omega_{max} \right)$, see Figure \ref{ReImp}, we can estimate $\bm{\omega_{n_0}}$ and hence reconstruct $\epsilon_0(z)$ as
\begin{equation*}
\epsilon_0(z) = \frac{\lambda_{n_{0}}}{\left(\lambda_{n_{0}} - 1 \right)} \, \epsilon_{p}(\bm{\omega_{n_0}}),
\end{equation*}
which is an approximate solution to the exact one defined, i.e. constructed, by  $\dfrac{\lambda_{n_{0}}}{\left(\lambda_{n_{0}} - 1 \right)} \, \epsilon_{p}(\omega_{n_0,\mathbb{C}})$.

\begin{figure}[H]
     \centering
     \begin{subfigure}[H]{0.45\textwidth}
         \centering
         \includegraphics[width=\textwidth]{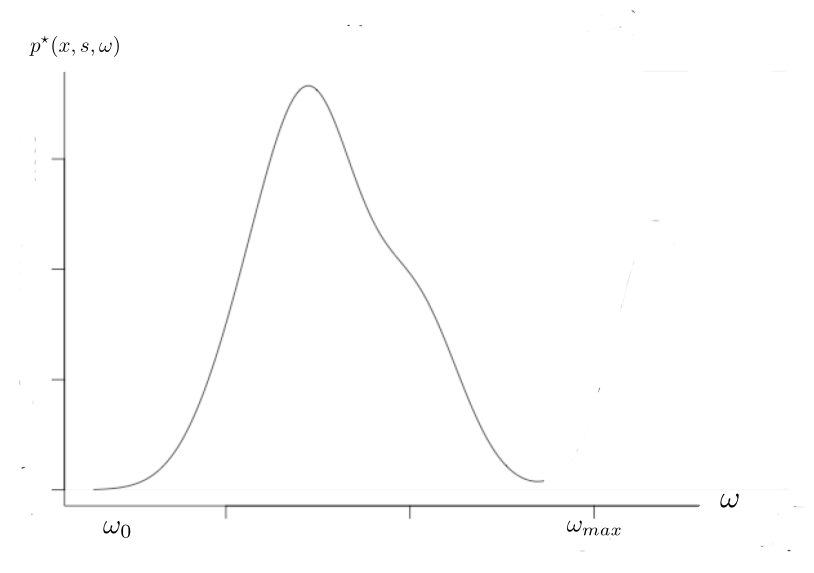} 
         \caption{A schematic representation of the function $\omega \rightarrow p^{\star}(\omega,x,s).$ The peak is reached for $\omega$ near $\bm{\omega_{n_{0}}}$.}
         \label{ReImp}
     \end{subfigure}
     \hfill
     \begin{subfigure}[H]{0.5\textwidth}
         \centering
\includegraphics[width=\textwidth]{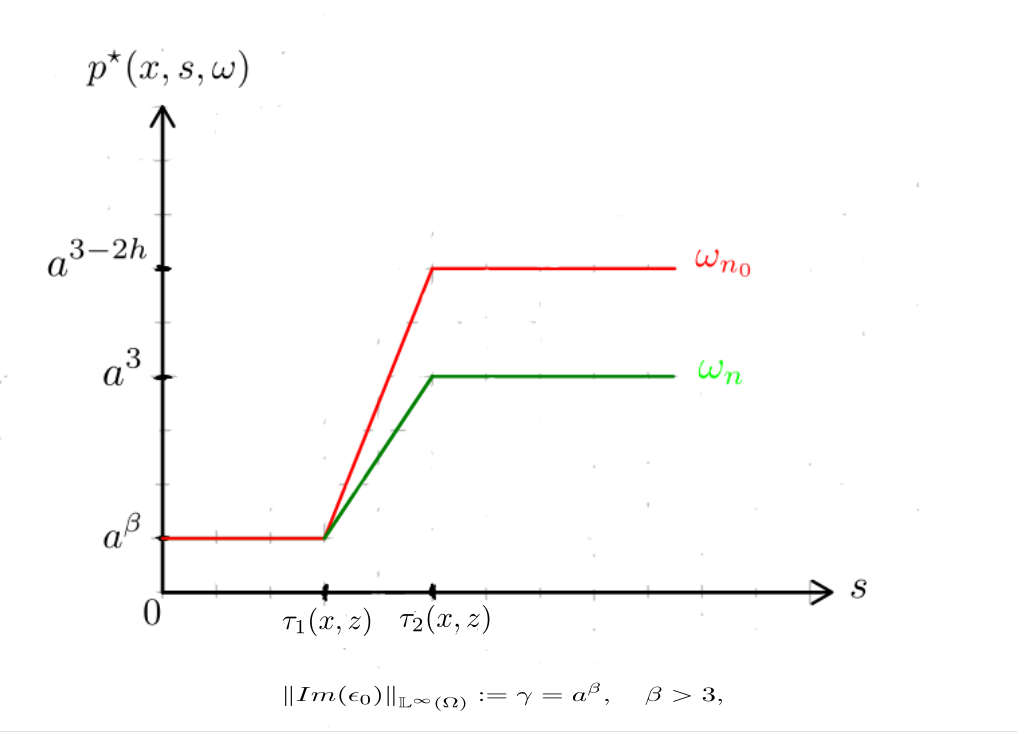}          
         \caption{Schematic representation for the average pressure $s \rightarrow p^{\star}(\omega,x,s)$. The case when we are away from the resonance is marked with green color. In this case we need $\beta>3$. The case when we are close to one resonance is marked with red color. In this case we can take $\beta > 3-h$.}
         \label{redgreen}
     \end{subfigure}
\end{figure}

\bigskip
\item Our next goal is to reconstruct the function $\rho$. To achieve this reconstruction, we see in Theorem \ref{THMVarSpeedCvar} that the measured pressure after  
the exit time, i.e. $ s > \tau_{2}(x, z) $, is given by the following expression
\begin{eqnarray*}
\int_{0}^{s} 2 r \int_{0}^{r} p(x,t)  dt \, dr &=&  \Im\left( \epsilon_{p} \right) \,   \,  \alpha_{-1}(z,x) \, \left\vert \underset{y}{\nabla} \tau(z,x) \right\vert \, \int_{D}   \left\vert u_{1} \right\vert^{2}(y) dy \\ &+&  \Im\left( \epsilon_{p} \right) \, \int_{\tau_{2}(x, z)}^{s} 2 \, r \, \sum_{k=0}^{+\infty} \alpha_{k}(z,x)  \,  \frac{\left(r^{2} - \tau^{2}(z,x) \right)^{k}}{k !} \,\, dr \, \int_{D}   \left\vert u_{1} \right\vert^{2}(y) dy  + Remainder_{1}.
\end{eqnarray*}
Now, in particular, when $s$ is close\footnote{We can prove that: 
\begin{equation*}
\left\vert \Im\left( \epsilon_{p} \right) \, \int_{\tau_{2}(x, z)}^{s} 2 \, r \, \sum_{k=0}^{+\infty} \alpha_{k}(z,x)  \,  \frac{\left(r^{2} - \tau^{2}(z,x) \right)^{k}}{k !} \,\, dr \,   \right\vert \lesssim \left\vert s - \tau_{2}(x, z) \right\vert.
\end{equation*}} to $\tau_{2}(x, z)$ the previous expression is reduced to 
\begin{equation*}
\int_{0}^{s} 2 r \int_{0}^{r} p(x,t)  dt \, dr = \Im\left( \epsilon_{p} \right) \,   \,  \alpha_{-1}(z,x) \, \left\vert \underset{y}{\nabla} \tau(z,x) \right\vert \, \int_{D}   \left\vert u_{1} \right\vert^{2}(y) dy   + Remainder_{2}.
\end{equation*}

Therefore, we can reconstruct $\Im\left( \epsilon_{p} \right) \,  \alpha_{-1}(z,x) \, \left\vert \nabla \tau(z,x) \right\vert \int_{D}   \left\vert u_{1} \right\vert^{2}(y) dy$. However, as we already reconstructed $\epsilon_{0}(\cdot)$ inside $\Omega$, then we can reconstruct $\int_{D}   \left\vert u_{1} \right\vert^{2}(y) dy$. As $\Im\left( \epsilon_{p} \right), \, \left\vert \nabla \tau(z,x) \right\vert$ and $\int_{D}   \left\vert u_{1} \right\vert^{2}(y) dy$ are known terms, we conclude the reconstruction of the function $\alpha_{-1}(z, x)$ and then by moving the plasmonic nano-particle inside $\Omega$, we reconstruct the function $\alpha_{-1}(\cdot, x)$ inside $\Omega$.

\bigskip

We see from $(\ref{defalpha-1})$, that the reconstruction of the functions $\tau(\cdot,\cdot)$ and $\alpha_{-1}(\cdot,\cdot)$ is sufficient, using numerical differentiation, to reconstruct the function $g_{x}(\cdot)$ defined, from $\overline{\Omega}$ to $\mathbb{R}$, by 
\begin{equation}\label{IGP}
g_{x}(y) := \int_{\Gamma(x,y)} \langle \underset{\xi}{\nabla} \log(\rho(\xi)) ; d\xi \rangle = \int_{\Gamma(x,y)}  \underset{\xi}{\bm{d}} \log(\rho(\xi)),  
\end{equation}
where $x \in \partial \Omega$ is fixed and, for every $y \in \overline{\Omega}$, the geodesic $\Gamma(x,y)$ is known, i.e constructed. The reconstruction of $\rho(\cdot)$ from $g_{x}(\cdot)$, in $(\ref{IGP})$, is an inverse integral problem along geodesics, which consists in determining a function by the given integrals of this function over a family of manifolds, see for instance \cite{amirov2014integral,
anikonov1997uniqueness, lavrent1967certain,romanov2013integral}. In our case, because $\underset{\xi}{\bm{d}} \log(\rho(\xi))$ is an exact one differential form, and most importantly the arriving points $y$ are (arbitrarily)  inside $\Omega$, the reconstruction of $\rho(\cdot)$ from the known function $g_{x}(\cdot)$ is quit easier and to accomplish this we proceed by a  
straightforward computation.
Let $\phi_{y}(\cdot)$ be $\mathcal{C}^{1}$-curve parametrizing the geodesic $\Gamma(\cdot,\cdot)$, i.e. 
\begin{eqnarray*}
\phi_{y}(\cdot) : [a,b_{y}] & \longrightarrow & \qquad \quad \,\, \overline{\Omega} \subset \mathbb{R}^{3} \\
s \quad & \longrightarrow & \left(\phi_{1,y}(s);\phi_{2,y}(s);\phi_{3,y}(s) \right), 
\end{eqnarray*} 
where  $a,b_{y} \in \mathbb{R}$ and
\begin{equation}\label{Gamma=phiab}
\phi_{y}(a) := x, \,\, \phi_{y}(b_{y}) := y \,\, \text{and} \,\,  \phi_{y}([a,b_{y}]) = \Gamma(x,y).
\end{equation}
With this parametrization, we obtain 
\begin{equation}\label{gxlogrho}
g_{x}(y) = \int_{\Gamma(x,y)}  \underset{\xi}{\bm{d}} \log(\rho(\xi)) \overset{(\star)}{=} \log(\rho(\phi_{y}(b_{y}))) - \log(\rho(\phi_{y}(a))) \overset{(\ref{Gamma=phiab})}{=} \log(\rho(y)) - \log(\rho(x)).
\end{equation}
For the justification of $(\star)$, in the previous formula, we refer the readers to  \cite[Proposition 3.4.1]{carton2007cours}. As the function $g_{x}(\cdot)$ is known we deduce from $(\ref{gxlogrho})$ the reconstruction of the function $\log(\rho(\cdot))$ up to an additive constant given by $-\log(\rho(x))$, where we recall that $x \in \partial \Omega$ is fixed. Consequently, we reconstruct the function $\rho(\cdot)$ up to a multiplicative constant given by $\dfrac{1}{\rho(x)}$. Furthermore, as we assumed that $\rho$ is smooth in $\mathbb{R}^3$ and we know $\rho$ outside $\Omega$ (to be constant), therefore $\rho(x)$, for $x\in \partial \Omega$ should be known (the known constant outside $\Omega$).

\end{enumerate}
\bigskip

 As we have reconstructed $\rho$ and $c$, then, in principle, we can also reconstruct the full  Green's kernel $G(\cdot,\cdot)$ by solving the forward problem for the wave equation. Actually, we can reconstruct $G(x,t, y)$, $x\in \partial \Omega$, $t>0$ and $y\in \Omega$,  directly from our measured data. To do this, we start by recalling its series expansion, see $(\ref{Green'skernelvarspeed})$,
\begin{equation*}
G(x,t,y) =  \sum_{k = -1}^{+\infty} \alpha_{k}(x,y) \,\, \Theta_{k}\left(t^{2}-\tau^{2}(x,y)\right), \quad x \neq y, \quad t \geq 0.
\end{equation*}
Since we have already constructed both the function  $\alpha_{-1}(\cdot,\cdot)$ and the travelling time $\tau(\cdot,\cdot)$, it is clear that to reconstruct $G(\cdot,\cdot)$ we need only to reconstruct the functions $\alpha_{k}(\cdot,\cdot), \; k \geq 0$.  From the definition of $\alpha_{0}(\cdot,\cdot)$, see $(\ref{defalphak})$, and taking into account that all the involved parameters $c(\cdot), \Gamma(\cdot,\cdot), \tau(\cdot,\cdot), \rho(\cdot)$ and $\alpha_{-1}(\cdot,\cdot)$ are already constructed, we can reconstruct $\alpha_{0}(\cdot,\cdot)$. Finally, by an induction process we reconstruct all the functions $\alpha_{k}(\cdot,\cdot)$, see (\ref{defalphak}).
\bigskip

In the previous literature dealing with the photo-acoustic imaging, we find uniqueness results and stability estimates of reconstructing  the variable speed $c$ or/and the initial source in \cite{S-U:2009, S-U-2013, S-Y-2017} under certain conditions. There are also numerical schemes proposed in \cite{Kirsch-Scherzer, N-S:2014} to reconstruct the speed/source assuming the speed to be piece-wise constant or enough small (under the Born approximation regime). In the current work, we proposed, using measurements generated by injected resonating nano-particles, a reconstruction scheme with which we can recover both the optical and acoustical properties of the object to image, i.e. permittivity, sound speed and mass density. The analysis is made under quite general conditions on these unknowns apart from the conditions in Hypotheses \ref{Hyp}. These last conditions appear naturally when dealing with inverse problems in integral geometry \cite{romanovBook, romanov2009smoothness, romanov2013integral}.

\bigskip

The remaining part of the manuscript is divided as follows. In Section 2, we give the proof of Theorem \ref{THMVarSpeedCvar}. In Section 3, we prove the existence and uniqueness of $p(\cdot,\cdot)$, solution of $(\ref{pressurwaveequa})$ with an $L^2$-source term, justify the integral representation  $(\ref{SolutionPressure})$ and the $L^4$-regularity of the source term $\dfrac{\omega \, \beta_{0}}{c_{p}} \; \Im(\varepsilon) \; \vert E \vert^{2}$ appearing in $(\ref{pressurwaveequa})$. Finally, we justify the point (2) of Section \ref{Section-Reconstruction} in Section \ref{Justification-omega-n-0}.
\bigskip

\section{Proof of Theorem \ref{THMVarSpeedCvar}}

We start by integrating $(\ref{SolutionPressureI})$, with respect to $t$ from $0$ to $r$, to obtain   
\begin{equation*}
\int_{0}^{r} p(x,t) \, dt = \int_{\Omega} G(x,r,y) \,\, \Im\left( \varepsilon \right)(y) \, \left\vert u_{1} \right\vert^{2}(y)  \,  dy - \int_{\Omega} G(x,0,y) \, \Im\left( \varepsilon \right)(y) \, \left\vert u_{1} \right\vert^{2}(y)  \,  dy.
\end{equation*}
For the second integral on the right hand side, from $(\ref{Green'skernelvarspeed})$, we have
\begin{equation*}
G(y,0,x)  =  \sum_{k = -1}^{+\infty} \alpha_{k}(y,x) \,\, \Theta_{k}\left(-\tau^{2}(y,x)\right) = 0, 
\end{equation*}
then 
\begin{eqnarray}\label{KNS1}
\nonumber
\int_{0}^{r} p(x,t) \, dt &=& \int_{\Omega} G(x,r,y) \,\, \Im\left( \varepsilon \right)(y) \, \left\vert u_{1} \right\vert^{2}(y)  \,  dy \\ \nonumber &\overset{(\ref{Green'skernelvarspeed})}{=}& \int_{\Omega}  \Im\left( \varepsilon \right)(y) \, \left\vert u_{1} \right\vert^{2}(y)  \,\, \alpha_{-1}(y,x) \,\, \delta\left(r^{2} - \tau^{2}(y,x)\right) \, dy \\
&+& \int_{\Omega} \Im\left( \varepsilon \right)(y) \, \left\vert u_{1} \right\vert^{2}(y)   \,\, \sum_{k = 0}^{+\infty} \alpha_{k}(y,x) \,\, \Theta_{k}\left(r^{2} - \tau^{2}(y,x)\right) \, dy. 
\end{eqnarray}
Now, using the fact that
\begin{equation*}
\delta\left(r^{2} - \tau^{2}(y,x)\right) = \frac{\delta\left(r - \tau(y,x)\right)+\delta\left(r + \tau(y,x)\right)}{2 \, r} = \frac{\delta\left(r - \tau(y,x)\right)}{2 \, r}
\end{equation*}
and the definition of the function $\Theta_{k}(\cdot)$, see $(\ref{defThetak})$, 
\begin{equation*}
\Theta_{k}\left(r^{2} - \tau^{2}(y,x)\right) = \frac{\left(r^{2} - \tau^{2}(y,x)\right)^{k}}{k !} \, \Theta_{0}\left(r^{2} - \tau^{2}(y,x)\right) = \frac{\left(r^{2} - \tau^{2}(y,x)\right)^{k}}{k !} \, \Theta_{0}\left(r - \tau(y,x)\right),
\end{equation*} 
we rewrite the formula $(\ref{KNS1})$ as
\begin{eqnarray}\label{KNS2}
\nonumber
\int_{0}^{r} p(x,t) \, dt &=& \frac{1}{2r} \, \int_{\Omega} \Im\left( \varepsilon \right)(y) \, \left\vert u_{1} \right\vert^{2}(y)   \alpha_{-1}(y,x) \,\, \delta\left(r - \tau(y,x)\right) \, dy \\
&+& \int_{\Omega} \Im\left( \varepsilon \right)(y) \, \left\vert u_{1} \right\vert^{2}(y)  \, \sum_{k = 0}^{+\infty} \alpha_{k}(y,x) \, \frac{\left(r^{2} - \tau^{2}(x,y) \right)^{k}}{k!} \,\Theta_{0}\left(r - \tau(y,x)\right) \, dy. 
\end{eqnarray}
To give sense to the first term on the right hand side, we need the relation, see  \cite[Chapter III, Section 1.3]{Gelfand1967GeneralizedFV},
\begin{equation*}
\int_{\Omega} f(y) \,\, \delta\left( g(y) \right) \,\, dy = \int_{g^{-1}(0)} f(y) \,\, d\sigma(y). 
\end{equation*}
Using this formula, we rewrite $(\ref{KNS2})$ as 
\begin{eqnarray}\label{ASB}
\nonumber
&&  2 r \int_{0}^{r} p(x,t)  dt =  \int_{\Omega \cap  \{ \tau(\cdot, x) = r \}}  \Im\left( \varepsilon \right)(y) \, \left\vert u_{1} \right\vert^{2}(y)  \, \alpha_{-1}(y,x)  \, d\sigma(y) \\
&+& 2 \, r \,  \int_{\Omega}  \Im\left( \varepsilon \right)(y) \, \left\vert u_{1} \right\vert^{2}(y)  \,\, \sum_{k = 0}^{+\infty} \alpha_{k}(y,x) \,\frac{\left(r^{2} - \tau^{2}(y,x)\right)^{k}}{k !} \, \Theta_{0}\left(r - \tau(y,x)\right) dy. 
\end{eqnarray}
Next, to write short formulas we need to fix some notations. For this, we set 
\begin{enumerate}
\item[$\ast$] Riemannian surface of center $x$ and radius $r$ will be denoted by
\begin{equation}\label{RS}
RS_{\tau}(x,r) := \{ y \in \mathbb{R}^{3} , \quad \text{such that} \quad \tau(x,y) = r \}. 
\end{equation} 
\item[$\ast$] Riemannian volume of center $x$ and radius $r$ will be denoted by 
\begin{equation}\label{RV}
RV_{\tau}(x,r) := \{ y \in \mathbb{R}^{3}, \quad \text{such that} \quad \tau(x,y) \leq r \}. 
\end{equation} 
\end{enumerate} 
With these notations, the equation $(\ref{ASB})$ takes the following form 
\begin{eqnarray*}
2 r \int_{0}^{r} p(x,t)  dt &=&   \int_{\Omega \cap RS_{\tau}(x,r)}  \Im\left( \varepsilon \right)(y) \, \left\vert u_{1} \right\vert^{2}(y)  \, \alpha_{-1}(y,x)  \, d\sigma(y) \\
&+& 2 \, r  \, \int_{\Omega \cap RV_{\tau}(x,r)}  \Im\left( \varepsilon \right)(y) \, \left\vert u_{1} \right\vert^{2}(y)  \, \sum_{k = 0}^{+\infty} \alpha_{k}(y,x) \,\frac{\left(r^{2} - \tau^{2}(y,x)\right)^{k}}{k !} \,  dy. 
\end{eqnarray*}
At this stage we integrate, with respect to the variable $r$, from $0$ to $s$ to obtain 
\begin{eqnarray}\label{EquaBCF}
\nonumber
\int_{0}^{s}  2 r \int_{0}^{r} p(x,t)  dt \, dr &=&  \int_{0}^{s} \int_{\Omega \cap RS_{\tau}(x,r)}  \Im\left( \varepsilon \right)(y) \, \left\vert u_{1} \right\vert^{2}(y)  \, \alpha_{-1}(y,x)  \, dy \, dr \\
&+& \int_{0}^{s} 2 \, r  \, \underset{\Omega \cap RV_{\tau}(x,r)}{\int}  \Im\left( \varepsilon \right)(y) \, \left\vert u_{1} \right\vert^{2}(y)  \, \sum_{k = 0}^{+\infty} \alpha_{k}(y,x) \,\frac{\left(r^{2} - \tau^{2}(y,x)\right)^{k}}{k !} \,  dy \, dr. 
\end{eqnarray}
For the first integral on the right hand side, we need the following lemma to transform the integral from a Riemannian surface to a Riemannian volume. 
\begin{lemma}\label{LemmaCoreaFormula}
We have, 
\begin{equation}\label{CoareaFormula}
\int_{0}^{s} \int_{\Omega \cap RS_{\tau}(x,r)} \, f(y) \, d\sigma(y) \, dr = \int_{\Omega \cap RV_{\tau}(x,s)} \, f(y) \, \left\vert \underset{y}{\nabla} \tau(x,y) \right\vert \, dy.
\end{equation}
\end{lemma}
\begin{proof}
See Subsection \ref{appendixtwo}.
\end{proof}
Thanks to $(\ref{CoareaFormula})$, we rewrite the formula $(\ref{EquaBCF})$ as 
\begin{eqnarray*}
\int_{0}^{s} 2 r \int_{0}^{r} p(x,t)  dt \, dr &=&  \int_{\Omega \cap RV_{\tau}(x,s)}  \Im\left( \varepsilon \right)(y) \, \left\vert u_{1} \right\vert^{2}(y)  \,  \alpha_{-1}(y,x) \, \left\vert \underset{y}{\nabla} \tau(x,y) \right\vert \, dy  \\
&+&  \int_{0}^{s} 2 \, r \int_{\Omega \cap RV_{\tau}(x,r)}  \Im\left( \varepsilon \right)(y) \, \left\vert u_{1} \right\vert^{2}(y)  \sum_{k = 0}^{+\infty} \alpha_{k}(y,x) \,\frac{\left(r^{2} - \tau^{2}(y,x)\right)^{k}}{k !} \,  dy \, dr. 
\end{eqnarray*}
We set $p^{\star}(x,s)$ to be 
\begin{equation*}
p^{\star}(x,s) := \int_{0}^{s}  2 r \int_{0}^{r} p(x,t)  dt \, dr,
\end{equation*} 
we split the domain $\Omega \cap RV_{\tau}(x,\cdot)$ into $(D \cap RV_{\tau}(x,\cdot)) \cup ((\Omega \setminus D) \cap RV_{\tau}(x,\cdot))$ then regarding the definition of the permittivity function $\varepsilon(\cdot)$, given by $(\ref{permittivityfct})$, we obtain 
\begin{eqnarray*}
\nonumber
p^{\star}(x,s) &=&  \Im\left( \epsilon_{p} \right) \, \int_{D \cap RV_{\tau}(x,s)}   \left\vert u_{1} \right\vert^{2}(y)  \,  \alpha_{-1}(y,x) \, \left\vert \underset{y}{\nabla} \tau(x,y) \right\vert  \, dy  \\ \nonumber
&+&  \Im\left( \epsilon_{p} \right) \, \int_{0}^{s} 2 \, r \int_{D \cap RV_{\tau}(x,r)}   \, \left\vert u_{1} \right\vert^{2}(y) \, \sum_{k = 0}^{+\infty} \alpha_{k}(y,x) \,\frac{\left(r^{2} - \tau^{2}(y,x)\right)^{k}}{k !} \,  dy \, dr \\ \nonumber
&+&  \int_{( \Omega \setminus D ) \cap RV_{\tau}(x,s)} \, \Im\left( \epsilon_{0} \right)(y) \, \left\vert u_{1} \right\vert^{2}(y)  \,  \alpha_{-1}(y,x) \, \left\vert \underset{y}{\nabla} \tau(x,y) \right\vert  \, dy  \\
&+& \int_{0}^{s} 2 \, r \int_{( \Omega \setminus D) \cap RV_{\tau}(x,r)}  \Im\left( \epsilon_{0} \right)(y) \, \left\vert u_{1} \right\vert^{2}(y)  \, \sum_{k = 0}^{+\infty} \alpha_{k}(y,x) \,\frac{\left(r^{2} - \tau^{2}(y,x)\right)^{k}}{k !} \,  dy \, dr. 
\end{eqnarray*}
To analyse the last two terms, of the previous formula, we need the following lemma.  
\begin{lemma}\label{LemmaMaxwell}
Let $f$ be sufficiently smooth function and $\delta$ positive parameter, then 
\begin{equation*}
\int_{( \Omega \setminus D ) \cap RV_{\tau}(x,\delta)} f(y) \, Im(\epsilon_{0})(y) \, \left\vert u_{1} \right\vert^{2}(y) \, dy = \int_{ \Omega  \cap RV_{\tau}(x,\delta)} f(y) \, Im(\epsilon_{0})(y) \, \left\vert u_{0} \right\vert^{2}(y) \, dy + \mathcal{O}\left( \bm{\gamma} \,\, a^{3-h} \right). 
\end{equation*} 
\end{lemma}
\begin{proof}
The proof is similar, taking into consideration the smallness of the imaginary part of the permittivity function, to the one given to prove Lemma 2.3 of the reference \cite{AhceneMouradMaxwell}.
\end{proof} 
Thanks to Lemma \ref{LemmaMaxwell}, we deduce that 
\begin{eqnarray*}
I_{3} &:=& \int_{( \Omega \setminus D ) \cap RV_{\tau}(x,s)}  \Im\left( \epsilon_{0} \right)(y) \, \left\vert u_{1} \right\vert^{2}(y)   \, \alpha_{-1}(y,x) \, \left\vert \underset{y}{\nabla} \tau(x,y) \right\vert \, dy \\ &=& \int_{\Omega  \cap RV_{\tau}(x,s)}  \Im\left( \epsilon_{0} \right)(y) \, \left\vert u_{0} \right\vert^{2}(y)   \, \alpha_{-1}(y,x)  \, \left\vert \underset{y}{\nabla} \tau(x,y) \right\vert \, dy + \mathcal{O}\left(\bm{\gamma} \,\, a^{3-h} \right) \\
& \text{and} & \\
I_{4} &:=& \int_{( \Omega \setminus D) \cap RV_{\tau}(x,r)}  \Im\left( \epsilon_{0} \right)(y) \, \left\vert u_{1} \right\vert^{2}(y)   \sum_{k = 0}^{+\infty} \alpha_{k}(y,x) \,\frac{\left(r^{2} - \tau^{2}(y,x)\right)^{k}}{k !} \,  dy \\ 
&=& \int_{\Omega \cap RV_{\tau}(x,r)}  \Im\left( \epsilon_{0} \right)(y) \, \left\vert u_{0} \right\vert^{2}(y)  \sum_{k = 0}^{+\infty} \alpha_{k}(y,x) \,\frac{\left(r^{2} - \tau^{2}(y,x)\right)^{k}}{k !} \,  dy +  \mathcal{O}\left( \bm{\gamma} \,\, a^{3-h} \right). 
\end{eqnarray*}
Consequently, 
\begin{eqnarray*}
p^{\star}(x,s) &=&  \Im\left( \epsilon_{p} \right) \, \int_{D \cap RV_{\tau}(x,s)}   \left\vert u_{1} \right\vert^{2}(y)  \,  \alpha_{-1}(y,x) \, \left\vert \underset{y}{\nabla} \tau(y,x) \right\vert   \, dy  \\
&+&  \Im\left( \epsilon_{p} \right) \, \int_{0}^{s} 2 \, r \int_{D \cap RV_{\tau}(x,r)}   \left\vert u_{1} \right\vert^{2}(y)  \sum_{k = 0}^{+\infty} \alpha_{k}(y,x) \,\frac{\left(r^{2} - \tau^{2}(y,x)\right)^{k}}{k !}  \, dy  \, dr \\
&+&  \int_{\Omega  \cap RV_{\tau}(x,s)} \Im\left( \epsilon_{0} \right)(y) \, \left\vert u_{0} \right\vert^{2}(y)  \,  \alpha_{-1}(y,x) \, \left\vert \underset{y}{\nabla} \tau(y,x) \right\vert    \, dy \\
&+&  \int_{0}^{s} 2 \, r \int_{\Omega \cap RV_{\tau}(x,r)}  \Im\left( \epsilon_{0} \right)(y) \, \left\vert u_{0} \right\vert^{2}(y)   \sum_{k = 0}^{+\infty} \alpha_{k}(y,x) \,\frac{\left(r^{2} - \tau^{2}(y,x)\right)^{k}}{k !} \,  dy  dr + \mathcal{O}\left(\bm{\gamma} \,\, a^{3-h} \right).
\end{eqnarray*}
Next, we set
\begin{eqnarray*}
p^{\star}_{0}(x,s) &:=&  \int_{\Omega  \cap RV_{\tau}(x,s)}  \Im\left( \epsilon_{0} \right)(y) \, \left\vert u_{0} \right\vert^{2}(y)  \,  \alpha_{-1}(y,x) \, \left\vert \underset{y}{\nabla} \tau(y,x) \right\vert    \, dy \\
&+&  \int_{0}^{s} 2 \, r \int_{\Omega \cap RV_{\tau}(x,r)}  \Im\left( \epsilon_{0} \right)(y) \, \left\vert u_{0} \right\vert^{2}(y)  \sum_{k = 0}^{+\infty} \alpha_{k}(y,x) \,\frac{\left(r^{2} - \tau^{2}(y,x)\right)^{k}}{k !} \,  dy  dr,
\end{eqnarray*}
which can be estimated at most as $\mathcal{O}\left( \bm{\gamma} \right)$. 
Then, 
\begin{eqnarray}\label{Lch}
\nonumber
p^{\star}(x,s) &=&  \Im\left( \epsilon_{p} \right) \, \int_{D \cap RV_{\tau}(x,s)}  \left\vert u_{1} \right\vert^{2}(y)  \,  \alpha_{-1}(y,x) \, \left\vert \underset{y}{\nabla} \tau(y,x) \right\vert  \, dy  \\
&+&  \Im\left( \epsilon_{p} \right) \, \int_{0}^{s} 2 \, r \int_{D \cap RV_{\tau}(x,r)}  \, \left\vert u_{1} \right\vert^{2}(y)   \sum_{k = 0}^{+\infty} \alpha_{k}(y,x) \,\frac{\left(r^{2} - \tau^{2}(y,x)\right)^{k}}{k !}  \,  dy \, dr + \mathcal{O}\left(\bm{\gamma}  \right).
\end{eqnarray}
To analyse the contribution of the injected nano-particle into the generated pressure, we must look at the behaviour of $p^{\star}(x,s)$, with respect to $s$, in three intervals of time. Before the entrance time, between the entrance time and the exit time and after the exit time. For this, we recall that
\begin{equation*}
\tau_{1}(x, z) = \left(\underset{y \in D}{Inf} \,\, \tau(x,y) \right) - a \quad \text{and} \quad \tau_{2}(x, z) =  \left(\underset{y \in D}{Sup} \,\, \tau(x,y) \right) + a.
\end{equation*}
Moreover, the smoothness of the function $\tau(x,\cdot)$ allows us to deduce 
\begin{equation}\label{tau1tau2}
\left\vert \tau_{1}(x, z) - \tau_{2}(x, z) \right\vert = \mathcal{O}\left( a \right).
\end{equation}
Now, regarding the value of $s$ we can distinguish three cases
\begin{enumerate}
\item[]
\item[$\ast$] Case when $s < \tau_{1}(x, z)$ (hence $r < \tau_{1}(x, z)$).
\item[] In this case $D \cap RV_{\tau}(x,s) \equiv D \cap RV_{\tau}(x,r) \equiv \{ \emptyset \}$ and, then, the equation $(\ref{Lch})$ will be reduced to the following formula 
\begin{equation*}
p^{\star}(x,s) = \mathcal{O}\left( \bm{\gamma} \right).
\end{equation*}
\item[]
\item[$\ast$] Case when  $\tau_{1}(x, z) < s < \tau_{2}(x, z)$.
\item[] In this case, first, we rewrite $(\ref{Lch})$ as 
\begin{eqnarray*}
\nonumber
p^{\star}(x,s) &=& \Im\left( \epsilon_{p} \right) \, \int_{D \cap RV_{\tau}(x,s)}  \left\vert u_{1} \right\vert^{2}(y)  \,  \alpha_{-1}(y,x) \, \left\vert \underset{y}{\nabla} \tau(y,x) \right\vert  \, dy  \\ \nonumber
&+&  \Im\left( \epsilon_{p} \right) \, \int_{0}^{\tau_{1}(x, z)} 2 \, r \int_{D \cap RV_{\tau}(x,r)}  \, \left\vert u_{1} \right\vert^{2}(y)  \sum_{k = 0}^{+\infty} \alpha_{k}(y,x) \,\frac{\left(r^{2} - \tau^{2}(y,x)\right)^{k}}{k !}  \,  dy \, dr \\ \nonumber &+&  \Im\left( \epsilon_{p} \right) \, \int_{\tau_{1}(x, z)}^{s} 2 \, r \int_{D \cap RV_{\tau}(x,r)}  \, \left\vert u_{1} \right\vert^{2}(y)   \sum_{k = 0}^{+\infty} \alpha_{k}(y,x) \,\frac{\left(r^{2} - \tau^{2}(y,x)\right)^{k}}{k !}  \,  dy \, dr + \mathcal{O}\left( \bm{\gamma} \right).
\end{eqnarray*}
Remark that in the second integral we have $D \cap RV_{\tau}(x,r) = \{ \emptyset \}$, then   
\begin{eqnarray*}
\nonumber
p^{\star}(x,s) &=& \Im\left( \epsilon_{p} \right) \, \int_{D \cap RV_{\tau}(x,s)}   \left\vert u_{1} \right\vert^{2}(y)  \,  \alpha_{-1}(y,x) \, \left\vert \underset{y}{\nabla} \tau(y,x) \right\vert  \, dy  \\ \nonumber
&+&  \Im\left( \epsilon_{p} \right) \, \int_{\tau_{1}(x, z)}^{s} 2 \, r \int_{D \cap RV_{\tau}(x,r)}  \, \left\vert u_{1} \right\vert^{2}(y)   \sum_{k = 0}^{+\infty} \alpha_{k}(y,x) \,\frac{\left(r^{2} - \tau^{2}(y,x)\right)^{k}}{k !}  \,  dy \, dr + \mathcal{O}\left( \bm{\gamma} \right).
\end{eqnarray*}
\begin{enumerate}
\item For the first integral, using Taylor expansion near the point $z^{\star} \in D \cap RV_{\tau}(x,s)$ and the a priori estimation given by $(\ref{aprioriestimate})$, we obtain 
\begin{eqnarray*}
I_{1} & := & \int_{D \cap RV_{\tau}(x,s)}  \left\vert u_{1} \right\vert^{2}(y)  \,  \alpha_{-1}(y,x) \, \left\vert \underset{y}{\nabla} \tau(y,x) \right\vert  \, dy \\ 
&=&  \alpha_{-1}(z^{\star} ,x) \, \left\vert \underset{y}{\nabla} \tau(z^{\star}, x) \right\vert \,\, \int_{D \cap RV_{\tau}(x,s)} \left\vert u_{1} \right\vert^{2}(y)  \,\, dy + \mathcal{O}\left( a^{4-2h} \right).
\end{eqnarray*}
\item For the second integral we have 
\begin{eqnarray*}
I_{2} &:=& \int_{\tau_{1}(x, z)}^{s} 2 \, r \int_{D \cap RV_{\tau}(x,r)}  \, \left\vert u_{1} \right\vert^{2}(y)  \sum_{k = 0}^{+\infty} \alpha_{k}(y,x) \,\frac{\left(r^{2} - \tau^{2}(y,x)\right)^{k}}{k !}  \,  dy \, dr \\
\left\vert I_{2} \right\vert & \leq & 2 \, s  \, \underset{y \in D \atop r \in [\tau_{1}(x, z);s]}{Sup} \left( \sum_{k = 0}^{+\infty} \left\vert \alpha_{k}(y,x) \right\vert \frac{\left\vert r^{2} - \tau^{2}(y,x)\right\vert^{k}}{k !}  \, \right) \, \left\vert s -  \tau_{1}(x, z) \right\vert \, \left\Vert u_{1} \right\Vert^{2}_{\mathbb{L}^{2}(D \cap RV_{\tau}(x,r))} \\
& \lesssim & \left\vert s -  \tau_{1}(x, z) \right\vert \, \left\Vert u_{1} \right\Vert^{2}_{\mathbb{L}^{2}(D)} = \mathcal{O}\left( a^{4-2h} \right), 
\end{eqnarray*}
where the last estimation is a consequence of $(\ref{tau1tau2})$ and $(\ref{aprioriestimate})$.
\end{enumerate}
Consequently, from $I_{1}$ and $I_{2}$, the expression of $p^{\star}(x,s)$ will be reduced to 
\begin{equation*}
p^{\star}(x,s) =  \Im\left( \epsilon_{p} \right)  \,    \alpha_{-1}(z^{\star},x) \, \left\vert \underset{y}{\nabla} \tau(z^{\star},x) \right\vert  \, \int_{D \cap RV_{\tau}(x,s)}  \left\vert u_{1} \right\vert^{2}(y) \, dy  +  \mathcal{O}\left( a^{4-2h} \right) + \mathcal{O}\left( \bm{\gamma} \right).
\end{equation*}
We set
\begin{equation*}
\Psi_{1}(x,z^{\star}) :=  \Im\left( \epsilon_{p} \right)  \,    \alpha_{-1}(z^{\star},x) \, \left\vert \underset{y}{\nabla} \tau(z^{\star},x) \right\vert
\end{equation*}
then 
\begin{equation*}
p^{\star}(x,s) = \Psi_{1}(x,z^{\star})  \, \int_{D \cap RV_{\tau}(x,s)}  \left\vert u_{1} \right\vert^{2}(y) \, dy  +  \mathcal{O}\left( a^{4-2h} \right) + \mathcal{O}\left( \bm{\gamma} \right).
\end{equation*}
\item[$\ast$] Case when $\tau_{2}(x, z) < s$. 
\item[] In this case, first, we rewrite $(\ref{Lch})$ as 
\begin{eqnarray*}
\nonumber
p^{\star}(x,s) &=&  \Im\left( \epsilon_{p} \right) \, \int_{D \cap RV_{\tau}(x,s)} \left\vert u_{1} \right\vert^{2}(y)  \,  \alpha_{-1}(y,x) \, \left\vert \underset{y}{\nabla} \tau(y,x) \right\vert  \, dy  \\ \nonumber
&+&  \Im\left( \epsilon_{p} \right) \, \int_{0}^{\tau_{1}(x, z)} 2 \, r \int_{D \cap RV_{\tau}(x,r)}  \, \left\vert u_{1} \right\vert^{2}(y)  \sum_{k = 0}^{+\infty} \alpha_{k}(y,x) \,\frac{\left(r^{2} - \tau^{2}(y,x)\right)^{k}}{k !}  \,  dy \, dr \\ \nonumber &+&  \Im\left( \epsilon_{p} \right) \, \int_{\tau_{1}(x, z)}^{\tau_{2}(x, z)} 2 \, r \int_{D \cap RV_{\tau}(x,r)}  \, \left\vert u_{1} \right\vert^{2}(y)   \sum_{k = 0}^{+\infty} \alpha_{k}(y,x) \,\frac{\left(r^{2} - \tau^{2}(y,x)\right)^{k}}{k !}  \,  dy \, dr \\
&+&  \Im\left( \epsilon_{p} \right) \, \int_{\tau_{2}(x, z)}^{s} 2 \, r \int_{D \cap RV_{\tau}(x,r)}  \, \left\vert u_{1} \right\vert^{2}(y) \sum_{k = 0}^{+\infty} \alpha_{k}(y,x) \,\frac{\left(r^{2} - \tau^{2}(y,x)\right)^{k}}{k !}  \,  dy \, dr + \mathcal{O}\left( \bm{\gamma} \right).
\end{eqnarray*}
Remark that in the second integral we have $D \cap RV_{\tau}(x,r) = \{ \emptyset \}$, for the first integral we have $D \cap RV_{\tau}(x,s) = D$ and for the fourth integral we have $D \cap RV_{\tau}(x,r) = D$, then   
\begin{eqnarray*}
\nonumber
p^{\star}(x,s) &=&  \Im\left( \epsilon_{p} \right) \, \int_{D}   \left\vert u_{1} \right\vert^{2}(y)  \,  \alpha_{-1}(y,x) \, \left\vert \underset{y}{\nabla} \tau(y,x) \right\vert  \, dy  \\ \nonumber
&+&  \Im\left( \epsilon_{p} \right) \, \int_{\tau_{1}(x, z)}^{\tau_{2}(x, z)} 2 \, r \int_{D \cap RV_{\tau}(x,r)}  \, \left\vert u_{1} \right\vert^{2}(y)  \sum_{k = 0}^{+\infty} \alpha_{k}(y,x) \,\frac{\left(r^{2} - \tau^{2}(y,x)\right)^{k}}{k !}  \,  dy \, dr \\
&+&  \Im\left( \epsilon_{p} \right) \, \int_{\tau_{2}(x, z)}^{s} 2 \, r \int_{D}  \, \left\vert u_{1} \right\vert^{2}(y)  \sum_{k = 0}^{+\infty} \alpha_{k}(y,x) \,\frac{\left(r^{2} - \tau^{2}(y,x)\right)^{k}}{k !}  \,  dy \, dr + \mathcal{O}\left( \bm{\gamma} \right).
\end{eqnarray*}
As done in the estimation of $I_{2}$, using the estimation $(\ref{tau1tau2})$ and $(\ref{aprioriestimate})$, we reduce the previous formula to 
\begin{eqnarray*}
\nonumber
p^{\star}(x,s) &=& \Im\left( \epsilon_{p} \right) \, \int_{D}   \left\vert u_{1} \right\vert^{2}(y)  \,  \alpha_{-1}(y,x) \, \left\vert \underset{y}{\nabla} \tau(y,x) \right\vert  \, dy  \\ \nonumber
&+&  \Im\left( \epsilon_{p} \right) \, \int_{\tau_{2}(x, z)}^{s} 2 \, r \int_{D}  \, \left\vert u_{1} \right\vert^{2}(y)   \sum_{k = 0}^{+\infty} \alpha_{k}(y,x) \,\frac{\left(r^{2} - \tau^{2}(y,x)\right)^{k}}{k !}  \,  dy \, dr + \mathcal{O}\left( a^{4-2h} \right) + \mathcal{O}\left( \bm{\gamma} \right).
\end{eqnarray*}
Now, using Taylor expansion near $z$, the center of the nano-particle, we obtain 
\begin{eqnarray*}
\nonumber
p^{\star}(x,s) &=&  \Im\left( \epsilon_{p} \right) \, \Bigg[ \,\,  \alpha_{-1}(z,x) \, \left\vert \underset{y}{\nabla} \tau(z,x) \right\vert \,   \\ \nonumber
&&  \qquad \qquad \qquad \quad + \int_{\tau_{2}(x, z)}^{s} 2 \, r \, \sum_{k = 0}^{+\infty} \alpha_{k}(z,x) \,\frac{\left(r^{2} - \tau^{2}(z,x)\right)^{k}}{k !} \, dr \Bigg] \int_{D}  \left\vert u_{1} \right\vert^{2}(y)   \,  dy  \\ \nonumber
&+& \mathcal{O}\left( a^{4-2h} \right) + \mathcal{O}\left( \bm{\gamma} \right).
\end{eqnarray*} 
We set 
\begin{equation*}
\Psi_{2}(x,z,s) :=  \Im\left( \epsilon_{p} \right) \, \left[  \alpha_{-1}(z,x) \, \left\vert \underset{y}{\nabla} \tau(z,x) \right\vert +   \int_{\tau_{2}(x, z)}^{s} 2 \, r \, \sum_{k = 0}^{+\infty} \alpha_{k}(z,x) \,\frac{\left(r^{2} - \tau^{2}(z,x)\right)^{k}}{k !} \, dr \right] 
\end{equation*}
then 
\begin{equation*}
p^{\star}(x,s) = \Psi_{2}(x,z,s) \,\, \int_{D} \left\vert u_{1} \right\vert^{2}(y)   \,  dy  + \mathcal{O}\left( a^{4-2h} \right) + \mathcal{O}\left( \bm{\gamma} \right).
\end{equation*}
\end{enumerate}
This ends the proof of Theorem $\ref{THMVarSpeedCvar}$. 
\section{Well-posedeness and integral representations}\label{Welposedness-integral-represention}
\subsection{Well-posedeness for the wave equation with $L^2$-source terms}\label{Eu} 
We start by recalling the acoustic model given by $(\ref{pressurwaveequa})$,   
\begin{equation}\label{Equa1App}
\left\{
\begin{array}{rll}
    \partial^{2}_{t} p(x,t) - c^{2}(x) \underset{x}{\Delta} p(x,t) &=& - c^{2}(x) \, \nabla \log(\rho(x)) \cdot \underset{x}{\nabla} p(x,t) \qquad in \quad \mathbb{R}^{3} \times \mathbb{R}^{+},\\
  \partial_{t}p(x,0) = 0 \quad \text{and} \quad  p(x,0) &=& \dfrac{\omega \, \beta_{0}}{c_{p}} \; \Im(\varepsilon)(x) \; \vert E \vert^{2}(x), \qquad in \quad \mathbb{R}^{3}. 
    \end{array}
\right.
\end{equation}
We denote by $H(\cdot,\cdot)$, the Green's kernel solution, in the distributional sense, of 
\begin{equation}\label{Equa2App}
\left\{
\begin{array}{rll}
    \partial^{2}_{t} H(x,t,y,s) - c^{2}(x) \underset{x}{\Delta} H(x,t,y,s) &=& \underset{y}{\delta}(x) \, \underset{s}{\delta}(t) \qquad in \quad \mathbb{R}^{3} \times \mathbb{R}^{+},\\
  \partial_{t}H(x,0) = H(x,0) &=& 0 \quad in \quad \mathbb{R}^{3}. 
    \end{array}
\right.
\end{equation}
By applying the Laplace transform to the equations $(\ref{Equa1App})$ and $(\ref{Equa2App})$, we obtain 
\begin{equation*}
\left\{
\begin{array}{rll}
    \underset{x}{\Delta} \, \mathcal{L}ap(p)(x,s) + \dfrac{(i \, s)^{2}}{c^{2}(x)}  \, \mathcal{L}ap(p)(x,s)  &=& \nabla \log(\rho(x)) \cdot \underset{x}{\nabla} \mathcal{L}ap(p)(x,s) - \dfrac{s}{c^{2}(x)} \, \dfrac{\omega \, \beta_{0}}{c_{p}} \; \Im(\varepsilon)(x) \; \vert E \vert^{2}(x), \\
    \underset{x}{\Delta} \mathcal{L}ap(H)(x,s) + \dfrac{(i \, s)^{2}}{c^{2}(x)}  \, \mathcal{L}ap(H)(x,s)  &=& \dfrac{-1}{c^{2}(x)} \, \underset{y}{\delta}(x)=\dfrac{-1}{c^{2}(y)} \, \underset{y}{\delta}(x). 
    \end{array}
\right.
\end{equation*}
Hence, we deduce that $\mathcal{L}ap(p)$ satisfies a Helmholtz equation with variable wave number and $\mathcal{L}ap(H)$, clearly, is its associated Green's kernel. Thanks to Green's formula we deduce that 
\begin{eqnarray*}
\mathcal{L}ap(p)(x,s) &=& s \; \mathcal{L}ap\left( \int_{\Omega} H(x,t,y) \dfrac{\omega \, \beta_{0}}{c_{p}} \; \Im(\varepsilon)(y) \; \vert E \vert^{2}(y) \, dy \right)(s) \\
&-& \int_{\Omega} \mathcal{L}ap(H)(x,s,y) \, c^{2}(y) \, \nabla \log(\rho(y)) \cdot \underset{y}{\nabla} \mathcal{L}ap(p)(y,s) \, dy.
\end{eqnarray*}
By taking the Laplace inverse transform, we deduce that 
\begin{eqnarray}\label{pressure=term+K}
\nonumber
p(x,t) &=& \partial_{t} \; \int_{\Omega} H(x,t,y) \dfrac{\omega \, \beta_{0}}{c_{p}} \; \Im(\varepsilon)(y) \; \vert E \vert^{2}(y) \, dy  \\
&-& \int_{\Omega} \, \int_{0}^{t} \, H(x,t-h,y) \, c^{2}(y) \, \nabla \log(\rho(y)) \cdot \underset{y}{\nabla} p(y,h) \, dh \, dy.
\end{eqnarray}
Next, we set $\bm{K}_{1}(\cdot)$ to be the operator defined by 
\begin{equation*}
\bm{K}_{1}(p(\cdot,t))(x) := \int_{\Omega} \, \int_{0}^{t} \, H(x,t-h,y) \, c^{2}(y) \, \nabla \log(\rho(y)) \cdot \underset{y}{\nabla} p(y,h) \, dh \, dy.
\end{equation*}
Then, 
\begin{equation*}
\left\Vert \bm{K}_{1}(p(\cdot,t)) \right\Vert^{2}_{\mathbb{L}^{2}(\Omega)} \leq \left\vert \Omega \right\vert \, t \, \int_{0}^{t} \, \int_{\Omega} \int_{\Omega} \,\, \left\vert H(x,t-h,y) \, c^{2}(y) \, \nabla \log(\rho(y)) \cdot \underset{y}{\nabla} p(y,h) \right\vert^{2} \, dy \, dx \, dh,
\end{equation*}
and using the continuity of the convolution  operator, with respect to the Helmholtz kernel $H(\cdot,t-h,\cdot)$, from $\mathbb{H}^{-1}(\Omega)$ to $\mathbb{H}^{1}(\Omega)$, we obtain:
\begin{equation*}
\left\Vert \bm{K}_{1}(p(\cdot,t)) \right\Vert^{2}_{\mathbb{L}^{2}(\Omega)} \leq \left\vert \Omega \right\vert \, t \, \left\Vert c^{2} \right\Vert^{2}_{\mathbb{L}^{\infty}(\Omega)} \, \left\Vert \nabla \log(\rho) \right\Vert^{2}_{\mathbb{L}^{\infty}(\Omega)}\, \int_{0}^{t} \,\, \left\Vert \nabla p(\cdot,h) \right\Vert^{2}_{\mathbb{H}^{-1}(\Omega)} \,\, dh,
\end{equation*}
hence, using the continuity of the gradient operator from $\mathbb{L}^{2}(\Omega)$ to $\mathbb{H}^{-1}(\Omega)$, we obtain: 
\begin{equation}\label{equap}
\left\Vert \bm{K}_{1}(p(\cdot,t)) \right\Vert^{2}_{\mathbb{L}^{2}(\Omega)} \leq \left\vert \Omega \right\vert \, t \, \left\Vert c^{2} \right\Vert^{2}_{\mathbb{L}^{\infty}(\Omega)} \, \left\Vert \nabla \log(\rho) \right\Vert^{2}_{\mathbb{L}^{\infty}(\Omega)} \, \left\Vert \nabla  \right\Vert^{2}_{\mathcal{L}} \, \int_{0}^{t} \,\, \left\Vert p(\cdot,h) \right\Vert^{2}_{\mathbb{L}^{2}(\Omega)} \,\, dh.
\end{equation}
Now, by taking the $\mathbb{L}^{2}(\Omega)$-norm in both sides of $(\ref{pressure=term+K})$ and gathering with $(\ref{equap})$ we obtain:
\begin{eqnarray*}
\nonumber
\left\Vert p(\cdot ,t) \right\Vert^{2}_{\mathbb{L}^{2}(\Omega)} & \leq & 2 \, \left\Vert \partial_{t} \; \int_{\Omega} H(\cdot,t,y) \dfrac{\omega \, \beta_{0}}{c_{p}} \; \Im(\varepsilon)(y) \; \vert E \vert^{2}(y) \, dy \right\Vert^{2}_{\mathbb{L}^{2}(\Omega)} \\
&+& 2 \, \left\vert \Omega \right\vert \, t \, \left\Vert c^{2} \right\Vert^{2}_{\mathbb{L}^{\infty}(\Omega)} \, \left\Vert \nabla \log(\rho) \right\Vert^{2}_{\mathbb{L}^{\infty}(\Omega)} \, \left\Vert \nabla  \right\Vert^{2}_{\mathcal{L}} \, \int_{0}^{t} \,\, \left\Vert p(\cdot,h) \right\Vert^{2}_{\mathbb{L}^{2}(\Omega)} \,\, dh .
\end{eqnarray*}
Thanks to the integral form of Grönwall Lemma, we deduce
\begin{equation*}
\left\Vert p(\cdot ,t) \right\Vert^{2}_{\mathbb{L}^{2}(\Omega)}  \leq  \bm{\alpha} \, \left\Vert \partial_{t} \; \int_{\Omega} H(\cdot,t,y) \dfrac{\omega \, \beta_{0}}{c_{p}} \; \Im(\varepsilon)(y) \; \vert E \vert^{2}(y) \, dy \right\Vert^{2}_{\mathbb{L}^{2}(\Omega)}, 
\end{equation*}
where
\begin{equation*}
\bm{\alpha} = 2 + 4 \, \left\vert \Omega \right\vert \, t \, \left\Vert c^{2} \right\Vert^{2}_{\mathbb{L}^{\infty}(\Omega)} \, \left\Vert \nabla \log(\rho) \right\Vert^{2}_{\mathbb{L}^{\infty}(\Omega)} \, \left\Vert \nabla  \right\Vert^{2}_{\mathcal{L}} \, e^{t}.
\end{equation*}
By referring to \cite[Theorem 7.1, formula 25]{salo}, we know 
that the operator $\bm{K}_{2}$, defined by the expression 
\begin{equation*}
\bm{K}_{2}\left( f \right)(x,t) := \partial_{t} \; \int_{\Omega} H(x,t,y) \, f(y) \, dy,
\end{equation*}
is a continuous operator from $\mathbb{L}^{2}(\Omega)$ to $\mathbb{L}^{2}(\Omega)$, with $\lambda_{2}$ as a constant of continuity, which is uniform in $t \in [0,T]$. Hence, 
\begin{equation*}
\left\Vert p(\cdot,t) \right\Vert_{\mathbb{L}^{2}(\Omega)} 
  \leq \lambda_{2} \; \bm{\alpha} \;  \dfrac{\omega \, \beta_{0}}{c_{p}} \;   \left\Vert \Im(\varepsilon) \right\Vert_{\mathbb{L}^{\infty}(\Omega)} \; \left\Vert  E \right\Vert^{2}_{\mathbb{L}^{4}(\Omega)}, \quad \forall \,\, t \in [0,T].
\end{equation*}
This concludes the proof. 
\subsection{Integral representation of the pressure, i.e.  proof of $(\ref{SolutionPressure})$.}\label{AppendixOne} 
By taking the Laplace transform with respect to time variable, that we denote in the sequel by $\mathcal{L}ap(\cdot)$, in both sides of $(\ref{pa2ndCGreen})$ and using the initial conditions, we obtain
\begin{equation*}
s^{2} \,\, \mathcal{L}ap(G)(x,s) - c^{2}(x) \,\, \mathcal{L}ap(\Delta \, G)(x,s) + c^{2}(x) \,\, \nabla \log(\rho(x)) \cdot \underset{x}{\nabla} \mathcal{L}ap(G)(x,s)=  \delta_{0}(x) \,\, \mathcal{L}ap(\delta_{0})(s).
\end{equation*}
From $\mathcal{L}ap(\Delta \, G) = \Delta \, \mathcal{L}ap(G)$, where $\Delta$ is taken with respect to spatial variable, we end up with the following equation  
\begin{equation}\label{LapG}
\underset{x}{\Delta} \,  \mathcal{L}ap(G)(x,s) -  \nabla \log(\rho(x)) \cdot \underset{x}{\nabla} \mathcal{L}ap(G)(x,s) + \frac{(i \, s)^{2}}{c^{2}(x)} \,\, \mathcal{L}ap(G)(x,s) = - \, c^{-2}(x) \, \delta_{0}(x).
\end{equation}
Analogous computations for the wave equation $(\ref{pressurwaveequa})$ allow us to get
\begin{equation}\label{Lapp}
\underset{x}{\Delta} \,  \mathcal{L}ap(p)(x,s) -  \nabla \log(\rho(x)) \cdot \underset{x}{\nabla} \mathcal{L}ap(p)(x,s) + \frac{(i \, s)^{2}}{c^{2}(x)} \,\, \mathcal{L}ap(p)(x,s) = - \,\, \frac{s}{c^{2}(x)} \,\, \frac{\omega \, \beta_{0}}{c_{p}} \,\, \Im\left( \varepsilon \right)(x) \, \left\vert E \right\vert^{2}(x) \, \underset{\Omega}{\chi}(x),
\end{equation}
By gathering $(\ref{LapG})$ and $(\ref{Lapp})$, we obtain the following representation
\begin{eqnarray*}
\mathcal{L}ap(p)(x,s) &=& \frac{\omega \, \beta_{0}}{c_{p}} \,\, s \,\, \int_{\Omega} \mathcal{L}ap(G)(x,y,s) \,\,  \Im\left( \varepsilon \right)(y) \, \left\vert E \right\vert^{2}(y) \,\, dy \\
&=&  \frac{\omega \, \beta_{0}}{c_{p}} \,\,  s \,\, \mathcal{L}ap\left( \int_{\Omega} G(x,y,t)  \,\, \Im\left( \varepsilon \right)(y) \, \left\vert E \right\vert^{2}(y) \,\, dy \right)(s). 
\end{eqnarray*}   
Again, using the fact that $G(x,0) = 0$, for $x \in \mathbb{R}^{3}$, and the Laplace transform we rewrite the previous formula as
\begin{equation*}
\mathcal{L}ap(p)(x,s) = \frac{\omega \, \beta_{0}}{c_{p}} \,\, \mathcal{L}ap\left(  \partial_{t} \,\, \int_{\Omega}  \, G(x,y,t)  \,\, \Im\left( \varepsilon \right)(y) \, \left\vert E \right\vert^{2}(y) \,\, dy \right)(s).
\end{equation*}
Finally, by taking the inverse Laplace  transform in both sides we obtain 
\begin{eqnarray*}
p(x,t) &=&  \frac{\omega \, \beta_{0}}{c_{p}} \,\, \partial_{t} \, \int_{\Omega}  \, G(x,y,t)  \,\, \Im\left( \varepsilon \right)(y) \, \left\vert E \right\vert^{2}(y) \,\, dy. 
\end{eqnarray*}
\subsection{Proof of Lemma $\ref{LemmaCoreaFormula}$.}\label{appendixtwo} First, 
\begin{eqnarray*}
\int_{0}^{s} \int_{\Omega \cap RS(x,r)} \, f(y) \, d\sigma(y) \, dr &=& \int_{0}^{s} \int_{RS(x,r)} \chi_{\Omega}(y) \, f(y) \, d\sigma(y) \, dr \\ &\overset{(\ref{RS})}{=}& \int_{0}^{s} \int_{\{ \tau(x,\cdot)=r \}} \chi_{\Omega}(y) \, f(y) \, d\sigma(y) \, dr  \\
&=& \int_{-\infty}^{+\infty} \chi_{[0,s]}(r) \int_{\{ \tau(x,\cdot)=r \}} \chi_{\Omega}(y) \, f(y) \, d\sigma(y) \, dr \\
&=& \int_{-\infty}^{+\infty} \left(  \int_{\{ \tau(x,\cdot)=r \}} \chi_{\Omega}(y) \, f(y) \, \chi_{\{ \tau(x,y) \leq s \}}(y) d\sigma(y) \right) dr, 
\end{eqnarray*}
then, using Coarea formula, see for instance Theorem 5 in Appendix C of \cite{Evans}, we obtain: 
\begin{eqnarray*}
\int_{0}^{s} \int_{\Omega \cap RS(x,r)} \, f(y) \, d\sigma(y) \, dr &=&   \int_{\mathbb{R}^{3}} \chi_{\Omega}(y) \, f(y) \, \chi_{\{ \tau(x,y) \leq s \}}(y) \, \left\vert \underset{y}{\nabla} \tau(x,y) \right\vert \, dy  \\
&=&   \int_{\Omega \cap \{ \tau(x, \cdot) \leq s \}}  \, f(y) \, \left\vert \underset{y}{\nabla} \tau(x,y) \right\vert \, dy  
\overset{(\ref{RV})}{=} \int_{\Omega \cap RV(x,s)}  \, f(y) \, \left\vert \underset{y}{\nabla} \tau(x,y) \right\vert \, dy.
\end{eqnarray*}
This ends the proof. 
\subsection{Regularity of the electric field $u_{1}:= E$}\label{33RegE}
The goal of this subsection is to justify the $\mathbb{L}^{2}(\mathbb{R}^{3})-$integrability of the source term given by $\dfrac{\omega \, \beta_{0}}{c_{p}} \, Im\left( \varepsilon \right) \, \left\vert E \right\vert^{2}$. Thanks to the fact that $Im\left( \varepsilon \right)$ is smooth function and vanishing one outside $\Omega$, we reduce the computations check to the $\mathbb{L}^{4}(\Omega)-$integrability of the electric field $u_{1}$ solution of the following Lippmann-Schwinger equation (L.S.E), 
\begin{equation*}
u_{1}(x) + \omega^{2} \, \mu \, \int_{D} \Pi(x,y) \cdot u_{1}(y) \, \left(\epsilon_{0}(y) - \epsilon_{p} \right) \, dy = u_{0}(x), \quad x \in \mathbb{R}^{3},
\end{equation*}
where $\Pi(\cdot,\cdot)$ is the Green tensor of the Maxwell problem with variable permittivity function.
In Theorem 2.1 of\cite{AhceneMouradMaxwell}, the previous integral representation was given sense. In addition, the coming decomposition of $\Pi(\cdot,\cdot)$ is proved
\begin{equation}\label{KernelDecomposition}
\Pi(x,y) =  \frac{1}{\omega^{2} \, \mu \, \epsilon_{0}(y)} \underset{y}{\nabla} \underset{y}{\nabla} \Phi_{0}(x,y) - \frac{1}{\omega^{2} \, \mu \, \left( \epsilon_{0}(y) \right)^{2}} \underset{x}{\nabla} \underset{x}{\nabla} M \left( \Phi_{0}(\cdot, y) \nabla \epsilon_{0}(y) \right)(x) + R(x,y), \quad x \neq y,
\end{equation}
where, for arbitrary small $\delta >0$, we have $R(\cdot,y) \in \mathbb{L}^{3-\delta}(D)$. Here $\Phi_{0}(x,y)$ is fundamental solution of the pure Laplacian, $\Phi_{0}(x,y):=\frac{1}{4\pi \vert x-y\vert}$. 
\begin{remark}\label{RemarkAppendix}
In $(\ref{KernelDecomposition})$, from singularity analysis point of view, we can approximate the second kernel, as 
\begin{equation*}
\underset{x}{\nabla} \underset{x}{\nabla} M \left( \Phi_{0}(\cdot, y) \nabla \epsilon_{0}(y) \right)(x) \simeq \underset{x}{\nabla} \left( \Phi_{0}(x, y) \nabla \epsilon_{0}(y) \right),
\end{equation*}
and the third kernel as   
\begin{equation*}
R(x,y) \simeq \Phi_{0}(x, y).
\end{equation*}
\end{remark}
Now, using the decomposition $(\ref{KernelDecomposition})$ the L.S.E becomes, 
\begin{eqnarray}\label{LSEE}
\nonumber
u_{1}(x) &-& \underset{x}{\nabla} \int_{D}   \underset{y}{\nabla} \Phi_{0}(x,y)  \cdot u_{1}(y)  \frac{\left(\epsilon_{0}(y) - \epsilon_{p} \right)}{\epsilon_{0}(y)}  \, dy  \\ \nonumber
 &-&  \, \int_{D}  \underset{x}{\nabla} \underset{x}{\nabla} M \left( \Phi_{0}(\cdot, y) \nabla \epsilon_{0}(y) \right)(x)  \cdot u_{1}(y) \, \frac{\left(\epsilon_{0}(y) - \epsilon_{p} \right)}{\left( \epsilon_{0}(y) \right)^{2}} \, dy  \\
 &+& \omega^{2} \, \mu \, \int_{D} R(x,y)  \cdot u_{1}(y) \, \left(\epsilon_{0}(y) - \epsilon_{p} \right) \, dy = u_{0}(x),
\end{eqnarray}
and thanks to Remark $\ref{RemarkAppendix}$, we can approximate the study of $(\ref{LSEE})$ by the following simplest equation 
\begin{equation}\label{LSEgradM}
u_{1}(x) - \nabla M \left( u_{1} \frac{\left(\epsilon_{0} - \epsilon_{p} \right)}{\epsilon_{0}} \right)(x) - \nabla N \left( \nabla \epsilon_{0} \cdot u_{1} \frac{\left(\epsilon_{0} - \epsilon_{p} \right)}{(\epsilon_{0})^{2}} \right)(x) + \omega^{2} \, \mu \, N \left( u_{1} \left(\epsilon_{0} - \epsilon_{p} \right) \right)(x) =  u_{0}(x),
\end{equation}
where $N(\cdot)$ is the Newtonian operator defined from $\mathbb{L}^{2}(D)$ to $\mathbb{L}^{2}(\Omega)$ by: 
\begin{equation*}
N(f)(x) := \int_{D} \Phi_{0}(x,y) \, f(y) \, d \, y = \int_{D} \frac{1}{\left\vert x - y \right\vert} \, f(y) \, d \, y.
\end{equation*}
To justify that $u_{1} \in \mathbb{L}^{4}(\Omega)$, we split our computation into two steps. First, we prove that $u_{1} \in \mathbb{L}^{4}(D)$. After restricting the equation $(\ref{eq:electromagnetic_scattering})$ into $D$, of course using also the fact that $\varepsilon(\cdot)_{|_{D}} := \epsilon_{p}$, see $(\ref{permittivityfct})$, we deduce that $\div(u_{1}) = 0$ and then,  by an integration by parts, we obtain:
\begin{equation}\label{IBPgradM}
\nabla M \left( u_{1} \frac{\left(\epsilon_{0} - \epsilon_{p} \right)}{\epsilon_{0}} \right) = - \nabla N \left( u_{1} \cdot \nabla  \frac{\left(\epsilon_{0} - \epsilon_{p} \right)}{\epsilon_{0}} \right) + \nabla S \left( \nu \cdot u_{1} \,\,  \frac{\left(\epsilon_{0} - \epsilon_{p} \right)}{\epsilon_{0}} \right),
\end{equation}
where $S(\cdot)$ is the Single-Layer operator with vanishing frequencies defined, from $\mathbb{L}^{2}(\partial B)$ to $\mathbb{H}^{\frac{3}{2}}(B)$, by:
\begin{equation*}
S(f)(x) := \int_{\partial D} \Phi_{0}(x,y) \, f(y) \, d \sigma (y) = \int_{\partial D} \frac{1}{\left\vert x - y \right\vert} \, f(y) \, d \sigma (y).
\end{equation*}
Hence, using $(\ref{IBPgradM})$, the equation $(\ref{LSEgradM})$ becomes,
\begin{eqnarray*}
u_{1}(x) - \nabla S \left( \nu \cdot u_{1}  \frac{\left(\epsilon_{0} - \epsilon_{p} \right)}{\epsilon_{0}} \right)(x) &=&  u_{0}(x) - \nabla N \left( u_{1} \cdot \nabla \frac{\left(\epsilon_{0} - \epsilon_{p} \right)}{\epsilon_{0}} \right)(x) \\ &-&  \omega^{2} \, \mu \, N \left( u_{1} \left(\epsilon_{0} - \epsilon_{p} \right) \right)(x) +  \nabla N \left( \nabla \epsilon_{0} \cdot u_{1} \frac{\left(\epsilon_{0} - \epsilon_{p} \right)}{(\epsilon_{0})^{2}} \right)(x), \quad x \in D. 
\end{eqnarray*}

To write a closed system of equations, we must take the normal derivative, from inside $D$, of the previous equation and we use the jump relations of the derivative of the Single-Layer operator to obtain: 
\begin{eqnarray}\label{BoundaryEqua}
\nonumber
\left[ \frac{\left(\epsilon_{0}(\cdot) + \epsilon_{p} \right)}{2 \, ( \epsilon_{0}(\cdot) - \epsilon_{p})} \; I -   K^{\star} \right] \, \left( \nu \cdot u_{1} \frac{( \epsilon_{0}(\cdot) - \epsilon_{p})}{\epsilon_{0}(\cdot)} \right)(x) &=& \nu \cdot u_{0}(x) - \omega^{2} \, \mu \, \nu \cdot N \left( u_{1} \left(\epsilon_{0} - \epsilon_{p} \right) \right)(x) \\ \nonumber &-&  \nu \cdot \nabla N \left( u_{1} \cdot \nabla  \frac{\left(\epsilon_{0} - \epsilon_{p} \right)}{\epsilon_{0}} \right)(x) \\ 
& + &  \nu \cdot \nabla N \left(\nabla \epsilon_{0} \cdot u_{1} \frac{(\epsilon_{0}-\epsilon_{p})}{(\epsilon_{0})^{2}} \right)(x), \qquad x \in \partial D,
\end{eqnarray}
where $K^{\star}(\cdot)$ is the Neumann-Poincar\'{e} operator defined and continuous, in the case where $\partial D$ is of class $ \mathcal{C}^{2}$ , from $\mathbb{H}^{-\frac{1}{2}}(\partial D)$ to $\mathbb{H}^{\frac{1}{2}}(\partial D)$, by\footnote{The notation $p.v$ refers to the Cauchy principal value.} 
\begin{equation*}
K^{\star}(f)(x) := p.v \int_{\partial D} \frac{\partial \;\; \Phi_{0}}{\partial \nu(x)}(x,y) \, f(y) \, d\sigma(y).
\end{equation*}
Then from $(\ref{BoundaryEqua})$, as the right hand side is a regular term, we get $\dfrac{\epsilon_{0}(\cdot) - \epsilon_{p}}{\epsilon_{0}(\cdot)} \, \nu \cdot u_{1} \in \mathbb{H}^{\frac{1}{2}}(\partial D)$ and then $\nu \cdot u_{1} \in \mathbb{H}^{\frac{1}{2}}(\partial D)$. Hence, $u_{1} \in \mathbb{L}^{2}(D), \,\, Curl(u_{1}) \in \mathbb{L}^{2}(D), \,\, \nu \cdot u_{1} \in \mathbb{H}^{\frac{1}{2}}(\partial D)$ and $\div(u_{1}) = 0$, then from \cite[Inequality (1.4)]{Amrouche-Houda} we deduce that  $u_{1} \in \mathbb{W}^{1,2}(D)$. Thanks to Rellich-Kondrachov theorem, see \cite[Theorem 9.16]{brezis2011functional}, we have the following compact injection $\mathbb{W}^{1,2}(D) \hookrightarrow \mathbb{L}^{4}(D)$. Therefore, $u_{1} \in \mathbb{L}^{4}(D).$ To show that $u_{1} \in \mathbb{L}^{4}(\Omega \setminus D)$, we rewrite $(\ref{LSEgradM})$ as 
\begin{equation*}
u_{1}(x) + \underset{x}{\nabla} \underset{x}{\nabla} N\left( u_{1} \frac{\left(\epsilon_{0} - \epsilon_{p} \right)}{\epsilon_{0}} \right)(x) - \nabla N \left( \nabla \epsilon_{0} \cdot u_{1} \frac{\left(\epsilon_{0} - \epsilon_{p} \right)}{(\epsilon_{0})^{2}} \right)(x) + \omega^{2} \, \mu \, N \left( u_{1} \left(\epsilon_{0} - \epsilon_{p} \right) \right)(x) =  u_{0}(x),
\end{equation*}
for $x \in \Omega$ and use the Calderon-Zygmund inequality, see for instance \cite[ page 242]{gilbarg2001elliptic},  to deduce that $u_{1} \in \mathbb{L}^{4}(\Omega)$.  

\section{Appendix. Estimation of the zeros of the dispersive equation $f_n(\omega, z)=0$}\label{Justification-omega-n-0}

For every $z\in \Omega$ and for every $n \in \mathbb{N}$, the dispersion equation 
\begin{equation}\label{DEqua}
f_n(\omega, z):=\epsilon_0(z) - (\epsilon_0(z)-\epsilon_{p}(\omega))\lambda_n=0,
\end{equation}
has one and only one solution in the complex plan, that we denote by $\omega_{n, \mathbb{C}}$, given by: 
\begin{equation}\label{solKAS}
\omega_{n, \mathbb{C}} = \frac{i \, \gamma_{p} \mp \sqrt{- \gamma_{p}^{2} + 4 \, \left( \omega^{2}_{0} + \dfrac{\epsilon_{\infty} \, \lambda_{n} \, \omega_{p}^{2}}{\epsilon_{\infty} \, \lambda_{n} + (1 - \lambda_{n}) \, \epsilon_{0}(z)} \right)}}{2},
\end{equation}
with it's dominant part contained in the interval $\left( \omega_{0};~~\sqrt{ \omega^{2}_{p} + \omega^{2}_{0}} \right)$. 
Our aims here is, for fixed $n_{0} \in \mathbb{N}$, to get an approximate root in the positive reel line to $\omega_{n_{0}, \mathbb{C}}$. To this end, we start by taking the Imaginary part in both sides of the dispersion equation given by $(\ref{DEqua})$, to obtain: 
\begin{eqnarray*}
\Im \left( f_{n_{0}}(\omega, z) \right) &=& \Im \left( \epsilon_0(z) - (\epsilon_0(z)-\epsilon_{p}(\omega))\lambda_{n_{0}} \right) \\
\Im \left( f_{n_{0}}(\omega, z) \right) &\overset{(\ref{plasmonic})}{=}& (1 - \lambda_{n_{0}}) \, \Im(\epsilon_{0}(z)) - \frac{\lambda_{n_{0}} \, \epsilon_{\infty} \, \omega_{p}^{2} \, \omega \, \gamma_{p}}{(\omega_{0}^{2} - \omega^{2})+(\omega \, \gamma_{p})^{2}} \\
\left\vert \Im \left( f_{n_{0}}(\omega, z) \right) \right\vert & \lesssim & \left\Vert \Im(\epsilon_{0}(\cdot)) \right\Vert_{\mathbb{L}^{\infty}(\Omega)} + \gamma_{p} \overset{(\ref{SELZ})}{=} \bm{\gamma}  + \mathcal{O}(\gamma_{p}), 
\end{eqnarray*} 
uniformly with respect to $\omega$. This implies the smallness of the imaginary part of the dispersion equation. Analogously, by taking the reel part in both sides of the dispersion equation  $(\ref{DEqua})$, we obtain: 
\begin{eqnarray*}
\Re \left( f_{n_{0}}(\omega, z) \right) &=& \Re \left( \epsilon_0(z) - (\epsilon_0(z)-\epsilon_{p}(\omega))\lambda_{n_{0}} \right) \\
\Re \left( f_{n_{0}}(\omega, z) \right) & \overset{(\ref{plasmonic})}{=}& (1 - \lambda_{n_{0}}) \, \Re(\epsilon_{0}(z)) + \lambda_{n_{0}} \, \epsilon_{\infty} \, \left[ 1 + \frac{\omega_{p}^{2} \, \left( \omega^{2}_{0} - \omega^{2} \right)}{(\omega_{0}^{2} - \omega^{2})+(\omega \, \gamma_{p})^{2}} \right],
\end{eqnarray*}
and investigating a solution for the equation $\Re \left( f_{n_{0}}(\omega, z) \right) = 0$, will be reduced to solve the quadratic equation, with respect to the unknown $\omega^{2}$, given by
\begin{eqnarray}\label{Equafreq}
\nonumber
\omega^{4} &+& \omega^{2} \left( - 2 \, \omega_{0}^{2} + \gamma_{p}^{2} - \frac{\omega_{p}^{2} \, \lambda_{n_{0}} \, \epsilon_{\infty}}{\lambda_{n_{0}} \, \epsilon_{\infty} + (1 - \lambda_{n_{0}}) \, \Re(\epsilon_{0}(z)) } \right) \\
&+& \omega_{0}^{4} +  \frac{\omega_{0}^{2} \, \omega_{p}^{2} \, \lambda_{n_{0}} \, \epsilon_{\infty}}{\lambda_{n_{0}} \, \epsilon_{\infty} + (1 - \lambda_{n_{0}}) \, \Re(\epsilon_{0}(z)) } = 0.
\end{eqnarray}
Now, by computing it's corresponding discriminant we have:
\begin{eqnarray}\label{Deltaapp}
\nonumber
\Delta &=& \left( - 2 \, \omega_{0}^{2} + \gamma_{p}^{2} - \frac{\omega_{p}^{2} \, \lambda_{n_{0}} \, \epsilon_{\infty}}{\lambda_{n_{0}} \, \epsilon_{\infty} + (1 - \lambda_{n_{0}}) \, \Re(\epsilon_{0}(z)) } \right)^{2} - 4 \, \left( \omega_{0}^{4} +  \frac{\omega_{0}^{2} \, \omega_{p}^{2} \, \lambda_{n_{0}} \, \epsilon_{\infty}}{\lambda_{n_{0}} \, \epsilon_{\infty} + (1 - \lambda_{n_{0}}) \, \Re(\epsilon_{0}(z)) } \right) \\
&=& \left( \frac{\omega_{p}^{2} \, \lambda_{n_{0}} \, \epsilon_{\infty}}{\lambda_{n_{0}} \, \epsilon_{\infty} + (1 - \lambda_{n_{0}}) \, \Re(\epsilon_{0}(z)) } \right)^{2} + \mathcal{O}\left(\gamma_{p}^{2} \right) > 0 .
\end{eqnarray}
Hence, by denoting $\omega^{2}_{n_{0}}$ the associated solution to $(\ref{Equafreq})$, we obtain: 
\begin{eqnarray*}
\omega^{2}_{n_{0}} &=& \frac{- \left( - 2 \, \omega_{0}^{2} + \gamma_{p}^{2} - \dfrac{\omega_{p}^{2} \, \lambda_{n_{0}} \, \epsilon_{\infty}}{\lambda_{n_{0}} \, \epsilon_{\infty} + (1 - \lambda_{n_{0}}) \, \Re(\epsilon_{0}(z)) } \right) \pm \sqrt{\Delta}}{2} \\
&\overset{(\ref{Deltaapp})}{=}& \frac{\left(  2 \, \omega_{0}^{2} - \gamma_{p}^{2} + \dfrac{\omega_{p}^{2} \, \lambda_{n_{0}} \, \epsilon_{\infty}}{\lambda_{n_{0}} \, \epsilon_{\infty} + (1 - \lambda_{n_{0}}) \, \Re(\epsilon_{0}(z)) } \right) \pm \left( \dfrac{\omega_{p}^{2} \, \lambda_{n_{0}} \, \epsilon_{\infty}}{\lambda_{n_{0}} \, \epsilon_{\infty} + (1 - \lambda_{n_{0}}) \, \Re(\epsilon_{0}(z)) } \right)}{2}  + \mathcal{O}\left(\gamma_{p}^{2} \right) \\
&=& \frac{\left(  2 \, \omega_{0}^{2} + \dfrac{\omega_{p}^{2} \, \lambda_{n_{0}} \, \epsilon_{\infty}}{\lambda_{n_{0}} \, \epsilon_{\infty} + (1 - \lambda_{n_{0}}) \, \Re(\epsilon_{0}(z)) } \right) \pm \left( \dfrac{\omega_{p}^{2} \, \lambda_{n_{0}} \, \epsilon_{\infty}}{\lambda_{n_{0}} \, \epsilon_{\infty} + (1 - \lambda_{n_{0}}) \, \Re(\epsilon_{0}(z)) } \right)}{2}  + \mathcal{O}\left(\gamma_{p}^{2} \right). 
\end{eqnarray*}
To avoid the trivial solution $\omega^{2}_{n_{0}} = \omega_{0}^{2} + \mathcal{O}\left(\gamma_{p}^{2} \right)$, we choose in the previous formula the positive sign to end with:
\begin{equation*}
\omega^{2}_{n_{0}} = \left( \omega_{0}^{2}  + \dfrac{\omega_{p}^{2} \, \lambda_{n_{0}} \, \epsilon_{\infty}}{\lambda_{n_{0}} \, \epsilon_{\infty} + (1 - \lambda_{n_{0}}) \, \Re(\epsilon_{0}(z)) } \right)  + \mathcal{O}\left(\gamma_{p}^{2} \right),
\end{equation*}
then 
\begin{equation*}
\omega_{n_{0}} = \left( \omega_{0}^{2}  + \dfrac{\omega_{p}^{2} \, \lambda_{n_{0}} \, \epsilon_{\infty}}{\lambda_{n_{0}} \, \epsilon_{\infty} + (1 - \lambda_{n_{0}}) \, \Re(\epsilon_{0}(z)) } \right)^{\frac{1}{2}}  + \mathcal{O}\left(\gamma_{p}^{2} \right).
\end{equation*} 
Next, for the previous solution, we set it's  dominant part to be: 
\begin{equation}\label{appsolKAS}
\bm{\omega_{n_{0}}} :=  \left( \omega_{0}^{2}  + \dfrac{\omega_{p}^{2} \, \lambda_{n_{0}} \, \epsilon_{\infty}}{\lambda_{n_{0}} \, \epsilon_{\infty} + (1 - \lambda_{n_{0}}) \, \Re(\epsilon_{0}(z)) } \right)^{\frac{1}{2}}.
\end{equation}
Additionally, in effortless manner,  we can check that
\begin{eqnarray*}
f_{n_{0}} \left(\bm{\omega_{n_{0}}}, z \right) &=& i \, \Im \left(\epsilon_{0}(z) \right) \, (1 - \lambda_{n_{0}}) + \mathcal{O}\left( \gamma_{p} \right) = \bm{\gamma} + \mathcal{O}\left( \gamma_{p} \right) \\
\left\vert f_{n_{0}}\left(\bm{\omega_{n_{0}}}, z \right) \right\vert & \leq & \left\Vert \Im \left(\epsilon_{0}(\cdot) \right) \right\Vert_{\mathbb{L}^{\infty}(\Omega)} + \mathcal{O}\left( \gamma_{p} \right) \overset{(\ref{SELZ})}{=} \bm{\gamma}  + \mathcal{O}(\gamma_{p}).
\end{eqnarray*}
Consequently, from  $(\ref{1Coro})$, we deduce that $p^{\star}(x,s,\cdot)$ admits a peaks near $\bm{\omega_{n_{0}}}$.
Therefore by plotting the curve $\omega \rightarrow p^{\star}(x,s,\omega)$, for $x \in \partial \Omega$ and $s> \tau_2(x, z)$ fixed, in the interval $\left( \omega_{0};~~\omega_{max} \right)$, see Figure \ref{ReImp}, we can estimate $\bm{\omega_{n_0}}$ and hence reconstruct $\epsilon_0(z)$ as
\begin{equation*}
\epsilon_0(z) = \frac{\lambda_{n_{0}}}{\left(\lambda_{n_{0}} - 1 \right)} \, \epsilon_{p}(\bm{\omega_{n_0}}),
\end{equation*}
which is an approximate solution to the exact one defined, i.e. constructed, by  $\dfrac{\lambda_{n_{0}}}{\left(\lambda_{n_{0}} - 1 \right)} \, \epsilon_{p}(\omega_{n_0,\mathbb{C}})$. More precisely we have, 
\begin{eqnarray*}
Error &:=& \frac{\lambda_{n_{0}}}{\left(\lambda_{n_{0}} - 1 \right)} \, \epsilon_{p}(\bm{\omega_{n_0}}) - \frac{\lambda_{n_{0}}}{\left(\lambda_{n_{0}} - 1 \right)} \, \epsilon_{p}(\omega_{{n_0},\mathbb{C}}) \\
Error & \overset{(\ref{plasmonic})}{=} & \frac{\lambda_{n_{0}} \, \epsilon_{\infty} \, \omega^{2}_{p}}{\left(\lambda_{n_{0}} - 1 \right)} \, \frac{\left(\bm{\omega_{n_0}}^{2} - \omega_{{n_0},\mathbb{C}}^{2} \right) + i \, \gamma_{p} \, \left(\omega_{{n_0},\mathbb{C}} - \bm{\omega_{n_0}} \right)}{\left(\omega^{2}_{0} - \bm{\omega_{n_0}}^{2} + i \, \gamma_{p} \,  \bm{\omega_{n_0}} \right) \, \left(\omega^{2}_{0} - \omega_{n_0,\mathbb{C}}^{2} + i \, \gamma_{p} \,  \omega_{n_0,\mathbb{C}} \right)} \\
\left\vert Error \right\vert & \lesssim & \left\vert \bm{\omega_{n_0}}^{2} - \omega_{{n_0},\mathbb{C}}^{2} \right\vert + \mathcal{O}\left( \gamma_{p} \right). 
\end{eqnarray*}
Then, using $(\ref{appsolKAS})$ and $(\ref{solKAS})$, we deduce
\begin{equation*}
\left\vert Error \right\vert  \lesssim  \Im(\epsilon_{0}(z)) + \mathcal{O}\left( \gamma_{p} \right) \leq \left\Vert  \Im(\epsilon_{0}(\cdot)) \right\Vert_{\mathbb{L}^{\infty}(\Omega)} + \mathcal{O}\left( \gamma_{p} \right)\overset{(\ref{SELZ})}{=} \bm{\gamma}  + \mathcal{O}(\gamma_{p}).
\end{equation*}

\thispagestyle{empty}
\mbox{}


\begin{thebibliography}{10}





\bibitem{amirov2014integral}
A. K. Amirov,
\newblock{Integral geometry and inverse problems for kinetic equations, De Gruyter, 2014.}


\bibitem{Habib-book}
H. Ammari,
\newblock{An introduction to mathematics of emerging biomedical imaging,}
\newblock{ Springer-Verlag, Volume 62, 2008.}


\bibitem{A-B-G-J:2012}
H. Ammari, E. Bretin, J. Garnier and V. Jugnon,
\newblock{Coherent Interferometry Algorithms for Photoacoustic Imaging}
\newblock{SIAM Journal on Numerical Analysis, Vol. 50, Iss. 5 (2012)10.1137/100814275.}







\bibitem{Amrouche-Houda}
C. Amrouche, N. El Houda Seloula. 
\newblock{$\mathbb{L}^{p}$-Theory for vector potentials and Sobolev's inequalities for vector fields. Application to the Stokes equations with pressure boundary conditions. 2011. hal-00686230.}





\bibitem{anikonov1997uniqueness}
Y. E. Anikonov and V. G. Romanov, 
\newblock{On uniqueness of determination of a form of first degree by its integrals along geodesics, Walter de Gruyter, 1997.}





\bibitem{B-B-M-T}
G.~Bal, E.~Bonnetier, F.~Monard, and F.~Triki.
\newblock Inverse diffusion from knowledge of power densities.
\newblock {\em Inverse Probl. Imaging}, 7(2):353-375, 2013.


\bibitem{B-U:2010}
G. Bal and G. Uhlmann,
\newblock Inverse diffusion theory of photoacoustics.
\newblock{\em Inverse Problems}, 26 (2010). 085010.


\bibitem{B-E-K-S:2018}
 A. Beigl, P. Elbau, K. Sadiq, O. Scherzer, 
\newblock{ Quantitative photoacoustic imaging in the acoustic regime using SPIM.}
\newblock{ Inverse Problems 34 (2018), no. 5, 054003, 5 pp.}


\bibitem{B-G-S:2016}
Z. Belhachmi, T. Glatz, O. Scherzer,
\newblock{ A direct method for photoacoustic tomography with inhomogeneous sound speed.}
\newblock{ Inverse Problems 32 (2016), no. 4, 045005, 25 pp.}


\bibitem{brezis2011functional}
H. Brezis,
\newblock{Functional analysis, Sobolev spaces and partial differential equations, volume 2,
number 3, Springer, 2011.}









\bibitem{carton2007cours}
H. Carton,
\newblock{Cours de calcul diff\'{e}rentiel, Herman, 2007.}














\bibitem{C-A-B:2007}
B. T. Cox, S. R. Arridge, P. C. Beard,
\newblock{Photoacoustic tomography with a limited-aperture planar sensor and a reverberant cavity}
\newblock{Inverse Problems, 2007 Inverse Problems 23 S95.}

\bibitem{Evans}
L.C. Evans,
\newblock{Partial Differential Equations, American Mathematical Society, 2010.}




  


\bibitem{Gelfand1967GeneralizedFV}
I. M. Gel'fand and G. E. Shilov, 
\newblock{Generalized Functions, Volume 1: Properties and Operations, American Mathematical Monthly, volume 377, pages 1026, 1967.}



\bibitem{AhceneMouradMaxwell}
A. Ghandriche and M. Sini,
\newblock{Photo-acoustic inversion using plasmonic contrast agents: The full Maxwell model, arXiv:2111.06269, 2021.}


\bibitem{gilbarg2001elliptic}
D. Gilbarg and N. S. Trudinger,
\newblock{Elliptic partial differential equations of second order, 2001, Springer.}

\bibitem{H-B-P-S}
M. Haltmeier, P. Burgholzer, G. Paltauf, and O. Scherzer,
\newblock{Thermoacoustic  computedtomography with large planar receivers,}
\newblock{ Inverse Problems, 20 (2004), pp. 1663-1673.}







\bibitem{Kirsch-Scherzer}
A. Kirsch and O. Scherzer, 
\newblock{Simultaneous reconstructions of absorption density and wave speed with photoacoustic measurements,} \newblock{SIAM J. Appl. Math. 72 (2012), no. 5, 1508-1523.}





\bibitem{K-K:2010} 
P. Kuchment and L. Kunyansky,
\newblock{Mathematics of thermoacoustic and photoacoustic tomography,} 
\newblock{in Handbook of Mathematical Methods in Imaging,
O. Scherzer, ed., Springer-Verlag,} pp. 817-866, 2010.


\bibitem{KuchmentKunyansky} 
P. Kuchment and L. Kunyansky,  
\newblock{Mathematics of thermoacoustic tomography, \,
European Journal of Applied Mathematics,}
\newblock{Volume 19, Number 02, 2008.}




\bibitem{lavrent1967certain}
M. M. Lavrent'ev and Yu. E. Anikonov, 
\newblock{A certain class of problems in integral geometry},
\newblock{Doklady Akademii Nauk, volume 176, number 5, pages 1002--1003, Russian Academy of Sciences, 1967.}







\bibitem{N-S:2014}
W. Naetar and O. Scherzer,
\newblock{Quantitative photoacoustic tomography with piecewise constant material parameters,}
\newblock{\em SIAM J.
Imag. Sci.,} V. 7, pp. 1755-1774, 2014.






\bibitem{P-P-B:2015}
A. Prost, F. Poisson and E. Bossy.
\newblock{ Photoacoustic generation by gold nanosphere: From linear to nonlinear thermoelastic in the 
long-pulse illumination regime}
\newblock arkiv:1501.04871v4









\bibitem{romanovBook}
V. G. Romanov,
\newblock{Inverse problems of mathematical physics, De Gruyter, 2018.}



\bibitem{romanov2009smoothness}
V. G. Romanov,
\newblock{On smoothness of a fundamental solution to a second order hyperbolic equation},
\newblock{Siberian Mathematical Journal, Springer,
volume 50, number 4, pages 700--705, 2009.}


\bibitem{romanov2013integral}
V. G. Romanov, 
\newblock{Integral geometry and inverse problems for hyperbolic equations,
  volume 26, Springer Science \& Business Media, 2013.}


\bibitem{salo}
M. Salo,
\newblock{Stability for solutions of wave equations with $C^{1, 1}$ coefficients, arXiv preprint math/0611457, 2006.}


\bibitem{S:2010}
O. Scherzer,
\newblock{Handbook of Mathematical Methods in Imaging,}
\newblock {\em Springer-Verlag}, 2010.


\bibitem{smith}
H. F. Smith,
\newblock{A parametrix construction for wave equations with $C^{1,1}$ coefficients, Annales de l'institut Fourier, vol. 48, num. 3, 797--835, 1998.}


\bibitem{S-U:2009}
P. Stefanov and G. Uhlmann,
\newblock{Thermoacoustic tomography with variable sound speed,} 
\newblock {\em Inverse Problems}, 25, 075011, 2009.
\bibitem{S-U-2013}
P. Stefanov and G. Uhlmann, 
\newblock{Recovery of a source term or a speed with one measurement and
applications.}
\newblock{Trans. Amer. Math. Soc., 365(11):5737-5758, 2013}


\bibitem{S-Y-2017}
P. Stefanov and Y. Yang,
\newblock{Thermo and Photoacoustic Tomography with variable speed and planar
detectors,}
\newblock{SIAM J. Math. Anal., 49(1), 297-310, 2017.}


\bibitem{Triki-Vauthrin:2017}
F. Triki and M. Vauthrin, 
\newblock{Mathematical modelization of the Photoacoustic effect generated by the heating of metallic nanoparticles,}
\newblock{Quart. Appl. Math. 76 (2018), 673-698.}






 



\end{thebibliography}
\end{document}